\documentclass{birkjour}
\usepackage{bbold}
\usepackage{graphicx}%
\usepackage{multirow}%
\usepackage{amsmath,amssymb,amsfonts}%
\usepackage{amsthm}%
\usepackage{mathrsfs}%
\usepackage[title]{appendix}%
\usepackage{xcolor}%
\usepackage{textcomp}%
\usepackage{manyfoot}%
\usepackage{booktabs}%
\usepackage{algorithm}%
\usepackage{algorithmicx}%
\usepackage{algpseudocode}%
\usepackage{listings}%
\usepackage{hyperref}

\newtheorem{theorem}{Theorem}[section]

\newtheorem{proposition}[theorem]{Proposition}

\newtheorem{remark}[theorem]{Remark}%
\newtheorem{definition}[theorem]{Definition}%

 \numberwithin{equation}{section}

\newcommand{\R}{{\mathbb{R}}}
\newcommand{\Rn}{{\mathbb{R}^d}}
\newcommand{\pfrac}{{(-\Delta)^s_p}}

\newcommand{\Wspz}{{W^{s,p}_0(\Omega)}}
\newcommand{\Wspzz}{{W^{s,p}_0(\Omega_0)}}

\newcommand{\dt}{{\Delta t}}

\newcommand{\vpz}{{\psi_{1,\Omega_0}}}
\newcommand{\pow}[2]{{\lceil#1\rceil^{#2}}}

\begin{document}

\title[Logistic elliptic and parabolic problems]{Logistic elliptic and parabolic problems involving the fractional $p$-Laplacian}

\author[L. Constantin]{Lo\"ic Constantin}

\address{%
LMAP, UMR E2S UPPA CNRS 5142\\
Universit\'e de Pau et Pays de l'Adour\\
B\^atiment IPRA, avenue de l'universit\'e, BP 1155\\
64013 Pau\\
FRANCE}

\email{loic.constantin@univ-pau.fr}
\author[C. A. Santos]{Carlos Alberto Santos}
\address{Department of Mathematics
Universidade de Brasilia\br
Federal district\br
70910-900 Brasilia\\
BRAZIL}
\email{csantos@unb.br}
\author[G. Warnault]{Guillaume Warnault}

\address{ (Corresponding author)\\
LMAP, UMR E2S UPPA CNRS 5142\\
Universit\'e de Pau et Pays de l'Adour\\
B\^atiment IPRA, avenue de l'universit\'e, BP 1155\\
64013 Pau\\
FRANCE}

\email{guillaume.warnault@univ-pau.fr}

\date{23th January 2026}

\begin{abstract}
In this paper, we deal with the questions of the existence and the uniqueness of weak solutions of the following nonlocal nonlinear logistic equation 
\begin{equation*}
\left\{ \begin{array}{l l}
\pfrac u_\lambda=\lambda u_\lambda^q  - b(x)u_\lambda^r \quad \text{in} \;\Omega,\\
u_\lambda=0 \quad \text{in} \; ( \Rn \backslash \Omega), \\
u_\lambda>0 \text{ in} \; \Omega.
\end{array}\right.
\end{equation*}
We also study convergence behaviors of $u_\lambda$ with respect to $\lambda$ underlining the effect of the nonlocal operator. We then look into the associated parabolic problem establishing local and global existences, uniqueness and asymptotic behaviors such as stabilization and blow up.
\end{abstract}

\maketitle
\section{Introduction}
We study the following nonlocal logistic problem in $\Omega$ a bounded regular domain of $\mathbb R^d$, $d\geq 2$,
\begin{equation*}\label{PbLE}\tag{$\mathcal E_{\lambda}$}
\left\{\begin{array}{l l}
\pfrac u=\lambda u^q  - b(x)u^r &\text{in} \;\Omega,\\
u=0 & \text{in} \; ( \Rn \backslash \Omega), \\
u>0 &\text{in} \; \Omega ,
\end{array}\right.
\end{equation*}
where:
\begin{itemize}
    \item the operator $\pfrac$ is the $p$-fractional Laplacian defined, up to a constant, for some $p\in(1,+\infty)$ and $s\in (0,1)$ by:
$$\pfrac u(x)= 2\,P.V. \int_\Rn \frac{|u(x)-u(y)|^{p-2}(u(x)-u(y))}{|x-y|^{d+sp}}dy$$
where $P.V.$ stands for Cauchy principal value;
\item  the parameters always satisfy $q>0$, $r> p-1$ and $\lambda>0$;
\item the function $b\not\equiv 0$ is bounded and nonnegative such that $b^{-1}(\{0\})=:\overline \Omega_0$ {where $\Omega_0$ is $C^{1,1}$-open set} and the Lebesgue measure of $\Omega_0$ is positive. 
\end{itemize}

The study of nonlocal operators has recently attracted increasing attention due to their occurrence in a wide range of physical phenomena characterized by long-range interactions. Elliptic and parabolic equations involving these operators, arise in diverse fields such as finance, physics, fluid dynamics, image processing, stochastic processes of L\'evy type, phase transitions, population dynamics, optimal control and game theory (see {\it e.g.} \cite{drapaca,silling,zimmermann}). 
Logistic equations have been vastly studied for their applications in mathematical biology. Indeed logistic equations can describe population dynamics, where the solution represents a population density. In this optic, taking the term $b=0$ on $\Omega_0$ can be seen as representing a positive effect on the population only on a subset of where the population lives.
\begin{equation*}
\left\{
\begin{array}{l l}
\pfrac u=f(x,u) &\text{in }\Omega,\\
u=0 &\text{in } ( \Rn \backslash \Omega),
\end{array}\right.
\end{equation*}
The previous elliptic problem has been studied for different source term in the literature see {\it e.g.} \cite{chen, hopf, iannizzotto, ian-mug}. As for the logistic problem the local case $s=1$ and $p=2$ has been studied in \cite{garcia,lopezgomez,santos,lopezgomez2} where the authors study the case $q=1$ to get existence, uniqueness and behavior results depending on $\lambda$. 
In the case of a linear operator, the nonlocal logistic equation has been extensively studied using various definitions of the fractional operator. For instance, the authors of  \cite{Montefusco} establish existence results of the logistic equation involving the half Laplacian defined by the spectral theory.  Similar results have be obtained in \cite{Carboni,Marinelli} considering the square root of $-\Delta$. With regard to the fractional Laplacian defined above by setting $p=2$, various logistic-type problems have been studied; see {\it e.g.} \cite{Caffarelli,Chhetri,Dwivedi,Khafagy,Marinelli2,Quaas}. In particular, the authors of \cite{Chhetri,Dwivedi,Khafagy} use sub- and supersolutions methods to prove the existence of solutions, furthermore a qualitative study proving bifurcation and symmetry results has been done in  \cite{Marinelli2}.\\
Our current work builds on the results of \cite{santos,Marinelli2} by extending them to the study of logistic equations involving nonlinear, nonlocal operators.
The main differences lie in the fact that the function $b$ vanishes in a subspace of $\Omega$ and in the role of the parameter $\lambda$ in the logistic term.
In the same way, the logistic problem \eqref{PbLE} has been processed in \cite{papageorgiou} for $b=1$ or in \cite{ian-mug} for more general nonlinearities addressing questions of existence, nonexistence and uniqueness.

In this article, we deal with the questions of existence, uniqueness and qualitative behaviors of the solution depending on the parameter $\lambda$.
More precisely, we use the Mountain Pass Theorem to establish the existence of a weak solution and then by Picone type inequality we prove the uniqueness as well as a comparison principle in the subhomogeneous case. We finish the study of the elliptic problem getting convergence properties as the parameter $\lambda$ goes to critical values depending to $p$ and $q$.
The main difference between the works of \cite{garcia,lopezgomez, santos,lopezgomez2} and ours in the case $q=p-1$ is the behavior of $u_\lambda$ depending on $\lambda$. In fact in these papers the authors prove the blow up on $\Omega_0$ and convergence to a degenerate solution on $\Omega\backslash\Omega_0$, whereas in our case, the nonlocal operator implies the blow up on the whole set $\Omega$.

We then study the parabolic problem associated to \eqref{PbLE}:
\begin{equation*}\label{pbp}\tag{$\mathcal P_\lambda$}
 \left\{ \begin{array}{l l}
     \partial_t u+ \pfrac u= \lambda u^q -b(x)u^r &\text{in }Q_T,\\
      u=0 &\text{in } (0,T)\times\Omega^c,\\
      u\geq 0 &\text{in } Q_T,\\
      u(0)=u_0 &\text{in } \Omega,
    \end{array}\right.
\end{equation*}
where $Q_T=(0,T)\times \Omega$, $u_0\in \Wspz\cap L^\infty(\Omega)$ is a nonnegative function and $\lambda>0$.
Parabolic problems involving the fractional $p$-Laplacian have been vastly studied in the literature, see {\it e.g.} \cite{abdellaoui,MR4913772,tiwari,mazon,vazquez}. In \cite{tiwari} the authors study the problem with $f(x,u)-g(x,u)$ as source term. The conditions they imposed on $f,g$ differs from our source term as well as the assumptions on the initial condition. The parabolic logistic equation has also been studied in the local case $s=1$, $p=2$, {\it e.g.} \cite{pardo}. The authors study global behavior of the parabolic logistic equation for $\Omega_0$ non-smooth. The problem \eqref{pbp} has also been study for $p=2$ in \cite{Chhetri2} where the authors prove the local existence using a sub- and supersolution method for a source term of the form $\lambda(a(x)u-bu^2-h(x))$ .
Up to our knowledge the parabolic logistic equation has never been studied in details for the fractional $p$-Laplacian case.
In this paper, we prove the existence of a bounded weak solution using a discretization method similarly as in \cite{MR4913772} and uniqueness from Gronwall  Lemma. Next, in the subhomogeneous case, a comparison principle and a sub-supersolution technique yield the convergence of a global solution of the parabolic problem to the solution of \eqref{PbLE}. Finally using energy method as well as Sattinger stable and unstable sets, we show blow up of our solution.

\section{Main results}
We first begin with preliminary results. We consider the fractional Sobolev space $\Wspz$ defined as follows
$$\Wspz= \big \{ u\in W^{s,p}(\Rn) \; | \; u=0 \text{ on } \Rn \backslash \Omega \big \},$$
with the Banach norm
$$\|u\|_\Wspz = \left( \int_{\Rn\times \Rn}  \frac{|u(x)-u(y)|^p}{|x-y|^{d+sp}}\,dxdy\right)^{\frac1p}$$
where $W^{s,p}(\Rn)=\big\{ u\in L^p(\Rn) \; | \; \|u\|_\Wspz<\infty \big\}$.\\
The set $\Wspz$ is a reflexive Banach space. If $sp<d$ the embeddings $\Wspz\hookrightarrow L^m(\Omega)$ are continuous for $m\in [1,\frac{dp}{d-sp}]$ and compact for $m<p^*:=\frac{dp}{d-sp}$ (see {\it e.g.} \cite[Corollary 7.2]{rando}). Furthermore, for $sp\geq d$,  $\Wspz\hookrightarrow L^r(\Omega)$ is compact for any $r\in [1,+\infty)$. For additional properties, we refer for instance \cite{bisci, rando}.\\
By the definition of the fractional $p$-Laplacian, we have for any $u,\,v \in \Wspz$:
$$\langle\pfrac u,v\rangle=\int_{\Rn\times \Rn} \frac{\pow{u(x)-u(y)}{p-1}(v(x)-v(y))}{|x-y|^{d+sp}}\,dx dy$$
where we use the notation $\pow{t}{p-1}=|t|^{p-2}t$.\\
Let $\lambda_1(\mathcal O)$ be the first eigenvalue for a general bounded regular domain $\mathcal O$
$$\lambda_1(\mathcal O)= \inf_{v\in W^{s,p}_0(\mathcal O) \backslash \{0\}}\frac{\|v\|^p_{W^{s,p}_0(\mathcal O)}}{\|v\|^p_{L^p(\mathcal O)}},$$
then we have that the mapping $\mathcal O\mapsto\lambda_1(\mathcal O)$ is decreasing in the sense that for $\mathcal O_1\subset\mathcal O_2$ such that the Lebesgue measure of $\mathcal O_2\backslash\mathcal O_1$ is positive, we have $\lambda_1(\mathcal O_1)>\lambda_1(\mathcal O_2)$.\\
In this work, we consider the notion of weak solution of problem \eqref{PbLE} defined as follows:
\begin{definition}[Weak solution]
We call a weak solution of \eqref{PbLE} a positive function $u \in \Wspz$ such that for any $\phi \in \Wspz$
\begin{equation}\label{ws}
\langle \pfrac u , \phi \rangle =\int_\Omega(\lambda u^q  - b(x)u^r)\phi\,dx.
\end{equation}
\end{definition}
\noindent For $q\leq p-1$, we introduce the set
$$\Lambda_q=\left\{ 
\begin{array}{l l}
    \mathbb (0,+\infty) & \mbox{if $q<p-1$} \\
     (\lambda_1(\Omega), \lambda_1(\Omega_0)) &  \mbox{if $q=p-1$.}
     \end{array}\right.$$
We then recall the existence theorem involving $\Lambda_q$:
\begin{theorem}\label{thexi1}
Problem \eqref{PbLE} admits a unique weak solution $u_\lambda$ {\it i.f.f.} $\lambda\in \Lambda_q$. Furthermore, for any $\lambda \in \Lambda_q$,  $u_\lambda$ belongs to $C^s(\Rn)$ satisfying $u_\lambda \geq cd(\cdot,\Omega^c)^s$ in $\Omega$ and for $\Lambda_q\ni\mu>\lambda $, 
we have $u_{\lambda}\leq u_{\mu}$ in $\Omega$.
\end{theorem}
The existence and the regularity come from  Theorem 1 and Lemma 2.3 in \cite{ian-mug} and from Theorem 1.5 in \cite{hopf}. The comparison principle is etablished by Proposition \ref{mono}.\\
We extend the previous existence result considering the superlinear case.
\begin{theorem}\label{thexi2}
Let $q\in (p-1, p^*-1)$ and $r< q$. For any $\lambda>0$,  \eqref{PbLE} admits at least a weak solution belonging to $C^s(\Rn)$. 
\end{theorem}

\begin{remark}
The case $r<q$ is not studied in the article \cite{papageorgiou} but the proof of Theorem \ref{thexi2} still holds for positive functions $b\in L^\infty(\Omega)$.
\end{remark}

The next result provides the uniform behavior of the sequence $(u_\lambda)_\lambda$ as $\lambda$ goes to the extremities of $\Lambda_q$:
\begin{theorem}\label{unifbu}
Let the sequence $(u_\lambda)_{\lambda\in \Lambda_q}$ defined by Theorem \ref{thexi1}. Then:
\begin{enumerate}
    \item [(i)] $(u_\lambda)_\lambda$ blows up uniformly on compact subsets of $\Omega$   (\it {i.e.}  for any $K$ compact subset of $\Omega$, $\min_K u_\lambda \to +\infty$) as $\lambda\to \sup \Lambda_q$ under the additional condition $p\geq 2$ if $q<p-1$.\\ More precisely, the sequence $(u_\lambda)_\lambda$ blows up uniformly on compact subsets of $\Omega_0$ in any case without additional hypothesis. 
    \item[(ii)] $(u_\lambda)_\lambda$ converges uniformly to $0$ as $\lambda\to \inf \Lambda_q$.
\end{enumerate}
\end{theorem}
For $q=p-1$, this results differs from the local case. In \cite{lopezgomez},  the authors study the problem \eqref{PbLE} involving the Laplacian operator, they obtain the blow-up only in $\Omega_0$ and the convergence to the minimal large positive solution of a singular problem on $\Omega\backslash\Omega_0$. Similar results are also obtained in the case of the $p$-Laplacian operator see \cite{duguo}. In our case, the nonlocality of the operator which considers the far effects implies the blow-up in $\Omega\backslash \Omega_0$.

We now state results on the associated parabolic problem \eqref{pbp} where we consider the class of weak solutions in the following framework 
$$X_T:=\Big\{v\in  L^\infty(Q_T)\, |\; \partial_tv\in L^2(Q_T)\Big\}\cap L^\infty([0,T];\Wspz)$$
and
\begin{definition}[Weak solution]\label{defsol2}
A nonnegative function $u \in X_T$ is a weak solution of \eqref{pbp} if $u(0,\cdot)=u_0$ {\it a.e.} in $\Omega$ such that for any $t\in [0, T]$:
\begin{equation}\label{fv2}
\int_0^t\int_\Omega \partial_tu \phi \,dxd\tau+\int_0^t\langle\pfrac u,\phi\rangle\,d\tau =\int_0^t\int_{\Omega}(\lambda u^q -b(x)u^r)\phi\,dxd\tau,
\end{equation}
for any $\phi \in L^1(0,T;\Wspz)\cap L^2(Q_T)$. 
\end{definition} 
We define the operator $\mathcal A: v\mapsto \mathcal Av=\pfrac v +b(x)|v|^{r-1}v$ and its domain
\begin{equation*}
D(\mathcal A)=\{ v\in \Wspz\cap L^\infty(\Omega)\;|\;\mathcal Av\in L^\infty(\Omega)\}.
\end{equation*}
\begin{remark}\label{pmbis}
Proposition 2.1 in \cite{MR4913772} implies that the operator $\mathcal A$ satisfies the comparison principle {\it i.e.} for any $u$, $v\in W^{s,p}(\Rn)$ such that $u\leq v$ {\it a.e.} in $\Rn\backslash \Omega$ and 
$$\langle \mathcal Au,\varphi\rangle\leq \langle \mathcal Av,\varphi\rangle$$
for any $\varphi \in \Wspz$, $\varphi\geq 0$. Then $u\leq v$ {\it a.e.} in $\Omega$.
\end{remark}
We have the following theorem of existence, uniqueness and regularity.
\begin{theorem}\label{mainpara}
Let $u_0\in \Wspz\cap L^\infty(\Omega)$ be a nonnegative function. Then, there exists $\mathcal T\in (0,+\infty]$ such that for any $T< \mathcal T$, \eqref{pbp} admits a weak solution $u$ in $Q_{T}$ belonging to $C([0,T];\Wspz)$. Additionally we have
\begin{enumerate}
    \item[(i)]  if $q\leq1$  or $\big(q\in(0, p-1] \text{ and }\lambda< \sup\Lambda_q\big)$ then   $\mathcal T=+\infty$. In this case, we say that the solution is global.
    \item[(ii)] if $q\geq 1$, then the weak solution is unique and for $u_0\in D(\mathcal A)$, $u$ belongs to $C([0,\mathcal T);C(\overline \Omega)).$
    \end{enumerate}
\end{theorem}
\begin{remark}\label{tmax}
When $q\geq1$, the uniqueness of the weak solution gives, as in Theorem 1.5 of \cite{MR4913772}, the existence of maximum life time of the solution: 
$$T_{max}:=\sup \{ T>0 \; | \;  u \text{ is the weak solution $u$ of \eqref{pbp} on }Q_T \}.$$
And if $T_{max}<\infty$ then $||u(t)||_{L^\infty(\Omega)}\to \infty$ as $t\to T_{max}$.
\end{remark}
\begin{theorem}\label{stab1}
Let $q\leq p-1$ and $\lambda\in \Lambda_q$. Let $u_0\in \Wspz\cap L^\infty(\Omega)$ such that $u_0\geq c d(\cdot, \tilde \Omega^c)^s$ for some $\tilde \Omega\subset \Omega$ a regular open set. In addition if $q=p-1$, we assume $\Omega_0  \subset\tilde \Omega$. Let $u$ be the global solution of \eqref{pbp} obtained by Theorem \ref{mainpara}. Then, $u(t)$ converges to $u_\lambda$ in $L^m(\Omega)$ for any $m\in [1,+\infty)$ as $t\to +\infty $. Furthermore if $q>1$ and $u_0\leq Cd(\cdot, \Omega^c)^s$ then $u(t)$ converges to $u_\lambda$ as $t\to +\infty $ in $L^\infty(\Omega)$. 
\end{theorem}

For the case $q=p-1$, the following theorem completes the asymptotic behavior given by Remark \ref{tmax} when $T_{max}=+\infty$  and given by the previous theorem.

\begin{theorem}\label{stabex}
Let $q=p-1$ and $\tilde \Omega\subset \Omega_0$ be a regular open set. Let $u_0\in \Wspz\cap L^\infty(\Omega)$ such that $u_0\geq c d(\cdot, \tilde \Omega^c)^s$.
If $\lambda\geq \lambda_1(\tilde \Omega)$ then $$\lim_{t\to \infty}\|u(t)\|_{L^\infty(\Omega)}=+\infty$$
where $u$ is the solution obtained by Theorem \ref{mainpara}.
\end{theorem}
In the same way, we establish 
\begin{proposition}\label{lemma}
Let $q=p-1$, let $\lambda> \lambda_1(\Omega_0)$ and $u_0\in \Wspz\cap L^\infty(\Omega)$ such that $u_0\geq cd(\cdot,\Omega_0^c)^s$. Let $u$ be the solution of \eqref{pbp} obtained by Theorem \ref{mainpara}. Then, $\|u(t)\|_{L^2(\Omega_0)}$ blows up as $t\to \infty$ for $q\leq1$ and  $T_{max}<+\infty$ for $q>1$.
\end{proposition}
Now introduce the following set
$$\mathcal H:=\{w\in \Wspz \ |\ \exists v_0\in \Wspzz \mbox{ {\it s.t.} } 0\leq v_0\leq w \mbox{ and } E_{\Omega_0}(v_0)<0\} $$
where for any $v\in \Wspzz$, we define the energy functional 
\begin{equation}\label{Eomega0}
    E_{\Omega_0}(v)=E(v,\Omega_0):=\frac{1}{p}\|v\|^p_\Wspzz-\frac{\lambda}{q+1}\|v\|^{q+1}_{L^{q+1}(\Omega_0)}.
\end{equation}

\begin{theorem}\label{q>p*-1}
Let $\lambda>0$, $q>p-1$ and $u_0\in \mathcal H \cap L^\infty(\Omega)$ be a nonnegative function. Then, the solution $u$  of \eqref{pbp} obtained by Theorem \ref{mainpara} satisfies
\begin{enumerate}
    \item[(i)] if $q>1$, then $T_{max}<+\infty,$
    \item[(ii)] if $q=1$, then $\|u(t)\|_{L^2(\Omega_0)}$ blows up as $t\to +\infty$,
    \item[(iii)] if $q<1$, then $\sup_{(0,+\infty)}\|u\|_{L^\infty(\Omega_0)}=+\infty$.
\end{enumerate}
\end{theorem}
In order to  extend Theorem \ref{q>p*-1}, 
we introduce the following sets. We define, for any open set $\mathcal O\subseteq \Omega$, the unstable set 
$$\mathcal U_{\mathcal O}:=\{v\in W^{s,p}_0(\mathcal O) \;|\;E_{\mathcal O}(v)<m_{\mathcal O} \mbox{ and } I_{\mathcal O}(v)<0\},$$
and the stable set
$$\mathcal S_{\mathcal O}:=\{v\in W^{s,p}_0(\mathcal O) \;|\;E_{\mathcal O}(v)<m_{\mathcal O} \mbox{ and } I_{\mathcal O}(v)>0\}\cup\{0\}$$
where 
$$m_{\mathcal O}=m(\mathcal O):=\inf_{v\in W^{s,p}_0(\mathcal O)\backslash\{0\}} \left(\sup_{\theta>0}E_{\mathcal O}(\theta v)\right)$$ and $$I_{\mathcal O}(v)=I(v,\mathcal O):=\|v\|^p_{W^{s,p}_0(\mathcal O)}-\lambda\|v\|^{q+1}_{L^{q+1}(\mathcal O)}.$$ 
And finally, we define the sets
$$ \mathcal H_u:=\{w\in \Wspzz \ | \ \exists v_0\in \mathcal U_{\Omega_0} \text{ s.t. } w\geq v_0\geq 0\}$$
and
$$\mathcal H_s:=\{w\in \Wspz \ | \ \exists v_0\in \mathcal S_{\Omega}\cap L^\infty(\Omega) \text{ s.t. } v_0\geq w\geq 0\}.$$
\begin{theorem}\label{exploextin}
Let $\lambda>0$, $q>p-1$ and $u_0\in \Wspz \cap L^\infty(\Omega)$. Let $u$ be the solution obtained by Theorem \ref{mainpara}. Then,
\begin{itemize}
\item if $u_0\in \mathcal H_u$, then {\it (i)}-{\it (iii)} of Theorem \ref{q>p*-1} hold.
\item if $u_0\in \mathcal H_s$ and $sp>d$ then the solution $u$ is global and $u(t)$ converges to $0$ in $L^m(\Omega)$ for any $m\in [1,+\infty)$ as $t\to +\infty$.
\end{itemize}
\end{theorem}

\begin{remark}
Using Theorem 1.7 of \cite{MR4913772} and a comparison principle, we get finite time extinction for $q\in (p-1,1]$ and $\|u_0\|_{L^\infty(\Omega)}$ small enough.
\end{remark}
Let us summarize the results and main contributions to the literature. 
\begin{itemize}
\item Theorem \ref{thexi2} generalizes the existence theorems of the elliptic logistic equation in the superlinear case.

\item Theorem \ref{unifbu} proves the qualitative behavior of the solution of \eqref{PbLE} depending on $\lambda$ highlighting the differences between the local and nonlocal cases. Indeed, for $q=p-1$, in the nonlocal case we obtain the uniform blow up on the entire space $\Omega$. 

\item About the parabolic problem, Theorems \ref{mainpara} to \ref{exploextin} generalize the parabolic $p$-fractional equation that has been studied for different source types (see {\it e.g.}\cite{abdellaoui,MR4913772,mazon,vazquez}). They also generalize the parabolic logistic equation that has been mainly in the local linear case in \cite{pardo2,pardo}.
\end{itemize}

The rest of the paper is organized as follows : in Section \ref{sec2}, we study \eqref{PbLE}, we prove  Theorem \ref{thexi2} in Section \ref{superl} and  the convergence properties of solutions with respect to $\lambda$ and to the power $q$ in Section \ref{sec21}.
Section \ref{sec3} is dedicated to the parabolic problem. More precisely we prove Theorem \ref{mainpara} in Section \ref{sec31} and Theorems \ref{stab1} to \ref{exploextin} in Section \ref{sec32}.

\section{Study of the elliptic problem \texorpdfstring{\eqref{PbLE}}{E} }\label{sec2}
We start with a comparison principle for problem \eqref{PbLE}. Let begin by the notion of sub- and supersolution.
\begin{definition}[Sub- and supersolution]
We say that $u$ is a subsolution (resp. supersolution) of \eqref{PbLE} if $u$ is a nonnegative function in $\Omega$ belonging to $W^{s,p}(\Rn)$ such that $u\leq 0$ (resp. $u\geq 0$) in $\Rn\backslash \Omega$ and satisfies  for any $\phi\in \Wspz$, $\phi \geq 0$:
$$ \langle\pfrac u, \phi \rangle \leq \int_\Omega (\lambda u^q -bu^r) \phi\,dx, \; \text{  (resp. }\geq).$$
\end{definition}

\begin{proposition}\label{mono}
Let $q\leq p-1$. Let $u,\,v\in \Wspz$ such that $u\geq c_1d(\cdot, \Omega^c)^s$, $v\leq c_2d(\cdot, \Omega^c)^s$
and are respectively super and subsolution of \eqref{PbLE}. Then $u \geq v$ {\it a.e.} in $\Omega$.
\end{proposition}

\begin{remark}\label{monob}
Proposition \ref{mono} still works for $b=0,$ thus we have a comparison principle for the subhomogeneous equation: $\pfrac u =\lambda u^q$ where $q\leq p-1.$
\end{remark}

\begin{proof}
We define $\Psi_\epsilon= (u_\epsilon-\frac{v_\epsilon^p}{u_\epsilon^{p-1 }})_-$, $\Phi_\epsilon=(v_\epsilon-\frac{u_\epsilon^p}{v_\epsilon^{p-1 }})_+$ in $\Omega$ and $\Psi_\epsilon=\Phi_\epsilon=0$ in $\Rn\backslash \Omega$ where $u_\epsilon=u+\epsilon$, $v_\epsilon=v+\epsilon$ and $f_+$, $f_-$ are respectively the positive and negative part of $f$ such that $f=f_++ f_-$.\\
The conditions on $u$ and $v$ imply  that  $\Psi_\epsilon,\, \Phi_\epsilon \in L^\infty (\Omega)\cap \Wspz$ hence we use $-\Psi_\epsilon\geq 0$ in the definition of the supersolution $u$ and in the same way $\Phi_\epsilon\geq0$ in the definition of the subsolution $v$ to obtain:
\begin{equation}\label{7}
\begin{split}
J_\epsilon:=\langle \pfrac u &,\Psi_\epsilon\rangle +\langle \pfrac v, \Phi_\epsilon\rangle\\
& \leq I_\epsilon:= \int_{\{u<v\}} (\lambda u^{q}-bu^r)\Psi_\epsilon+(\lambda v^{q}-bv^r)\Phi_\epsilon\,dx.
\end{split}
\end{equation}
We have, for any $\epsilon>0$
\begin{equation*}
J_\epsilon=  \int_{\{u<v\}^2}W(x,y)\,dxdy+ 2\int_{\{u<v\}\times\{u>v\}}\widetilde W(x,y)\,dxdy
\end{equation*}
where
\begin{equation*}
    \begin{split}
W(x,y):=&\pow{v_\epsilon(x)-v_\epsilon(y)}{p-1} \Big((v_\epsilon- \frac{u_\epsilon^p}{v_\epsilon^{p-1}})(x)-(v_\epsilon-\frac{u_\epsilon^p}{v_\epsilon^{p-1}})(y)\Big)\\
&+ \pow{u_\epsilon(x)-u_\epsilon(y)}{p-1} \Big((u_\epsilon-\frac{v_\epsilon^p}{u_\epsilon^{p-1}})(x)-(u_\epsilon-\frac{v_\epsilon^p}{u_\epsilon^{p-1}})(y)\Big)\geq 0
\end{split}
\end{equation*}
by Proposition 4.2 in \cite{picone} and 
\begin{equation*}
\begin{split}
\widetilde W(x,y):=&\pow{u_\epsilon(x)-u_\epsilon(y)}{p-1} (u_\epsilon-\frac{v_\epsilon^p}{u_\epsilon^{p-1}})(x)\\
&+\pow{v_\epsilon(x)-v_\epsilon(y)}{p-1} (v_\epsilon-\frac{u_\epsilon^p}{v_\epsilon^{p-1}})(x) \\
=&\left(\pow{1-\frac{u_\epsilon(y)}{u_\epsilon(x)}}{p-1} -\pow{1-\frac{v_\epsilon(y)}{v_\epsilon(x)}}{p-1}\right)(u_\epsilon^p(x)-v_\epsilon^p(x))\geq 0
\end{split}
\end{equation*}
since $0\leq \frac{u_\epsilon(y)}{u_\epsilon(x)}\leq \frac{u_\epsilon(y)}{u_\epsilon(x)}$ on $\{u<v\}\times\{u>v\}$ and the mapping $z\mapsto \pow{1-z}{p-1}$ is decreasing on $\R^+$.\\
Hence $J_\epsilon\geq 0$ and by \eqref{7}, we also have $I_\epsilon\geq 0$.\\
The conditions on $u$ and $v$ insure that the functions $\frac{u}{v}$ and $\frac{v}{u}$ are bounded on $\{u<v\}$ thus from the Dominated Convergence Theorem, we pass to the limit as $\epsilon\to 0$ and we get
$$I_\epsilon \to\int_{\{u<v\}} \lambda(u^p-v^p)(u^{q-p+1}-v^{q-p+1}) -b (u^p-v^p)(u^{r-p+1}-v^{r-p+1})\,dx.$$
Assume that the measure of $\{u<v\}$ is positive, then the right-hand side in the previous limit is negative  since $q\leq p-1<r$, which contradicts $I_\epsilon\geq 0$ for any $\epsilon>0$. Thus we deduce that $u\geq v$ {\it a.e.} in $\Omega$.
\end{proof}
\subsection{The superlinear case}\label{superl}
\begin{proof}[Proof of Theorem \ref{thexi2}]
We use the Mountain Pass Theorem see for instance \cite[Theorem 2.4]{mountainpass}. We define, for any $v\in \Wspz$:

$$J(v):= \frac{1}{p} \|v\|^p_\Wspz - \frac{\lambda}{q+1}\|v\|^{q+1}_{L^{q+1}(\Omega)}+\frac{1}{r+1}\int_\Omega b |v|^{r+1}\,dx.$$
Then,  $J$ belongs to $C^1(\Wspz, \mathbb{R})$ and
$$J'(u).v=\langle \pfrac u, v\rangle -\lambda \int_\Omega |u|^{q-1}uv\,dx + \int_\Omega b|u|^{r-1}uv\,dx.$$
Since $q+1<p^*$ and $b\geq 0$, there exists $C>0$ such that 
$$J(v)\geq  
\|v\|^p_\Wspz\left(\frac{1}{p} -C \|v\|^{q+1-p}_\Wspz\right). $$
Hence from $q+1>p$, we have that $J(v)> 0$  for any $\|v\|_\Wspz\neq 0$ small enough and we deduce that $0_\Wspz$ is a local minimizer of $J$.\\
Moreover, let $\phi \in \Wspzz$, then for $t$ large enough
$$J(t\phi) = t^p \frac{1}{p} \|\phi\|^p_\Wspz - t^{q+1} \frac{\lambda}{q+1}\|\phi\|^{q+1}_{L^{q+1}(\Omega)}<0. $$
By continuity, there exists $t_0>0$ such that $J( t_0 \phi)=0$.\\
Let now show that $J$ satisfies the Palais-Smale condition. Let $(u_n)_n\subset \Wspz$ such that $|J(u_n)|\leq C$ and $J'(u_n)$ converges to $0$ as $n \to \infty$. We have in particular, for $n$ large enough, $|J'(u_n).u_n|\leq ||u_n||_\Wspz$ hence
\begin{equation*}\label{5}
-\|u_n\|^p_\Wspz +\lambda \|u_n \|^{q+1}_{L^{q+1}(\Omega)}-\int_\Omega b |u_n|^{r+1} \leq \|u_n\|_\Wspz.
\end{equation*}
and from $|J(u_n)|\leq C$, we get 
\begin{equation*}\label{6}
\frac{1}{p} \|u_n\|^p_\Wspz +\frac{1}{r+1}\int_\Omega b |u_n|^{r+1} \leq C+ \frac{\lambda}{q+1}\|u_n\|^{q+1}_{L^{q+1}(\Omega)}.
\end{equation*}
Combining the both previous inequalities, we obtain
$$(\frac{q+1}{p}-1)\|u_n\|^p_\Wspz +(\frac{q+1}{r+1}-1)\int_\Omega b {|u_n|^{r+1}}\leq \|u_n\|_\Wspz+C'.$$
Since $q+1>p$ and $q>r$, we deduce than the sequence $(u_n)$ is bounded in $\Wspz$. Thus $u_n \rightharpoonup u$ in $\Wspz$ up to a subsequence and from the compact embeddings involving $\Wspz$, $q\,,r < p^*-1$, we have
$$-\lambda \int_\Omega |u_n|^{q-1}u_n(u_n-u) + \int_\Omega b|u_n|^{r-1}u_n(u_n-u)\to 0.$$
Since $J'(u_n).(u_n-u)\to 0$, we infer that 
$$\langle \pfrac u_n, u_n-u\rangle\to 0$$
combined with the weak convergence of  $(u_n)$ in $\Wspz$, we get 
$$\langle \pfrac u_n-\pfrac u, u_n-u\rangle\to 0$$
and it follows  $u_n\to u$ in $\Wspz$. Thus $J$ verifies the Palay-Smale condition and there exists $u\in \Wspz$ such that $J'(u)=0$. Theorem 3.1 of \cite{iannizzotto} and Theorem 2.7 of \cite{hold} insure that 
$u\in C^s(\Rn)$.\\ 
Its remains to prove that $u>0$ in $\Omega$. Since $J(u)>0$, we have that $u\not\equiv  0$ and by definition of $u$, we have
$$J(u)=\inf_{\gamma \in \Gamma} \left(\max_{s\in [0,1]} J(\gamma(s))\right),$$
where $\Gamma=\{\gamma\in C([0,1],\Wspz), \;\gamma(0)=0, \;J (\gamma(1))<0\}$.\\
Noting that, for any $v\in \Wspz$,  $\|v\|_\Wspz\geq \||v|\|_\Wspz$ 
, we deduce that on $\Gamma_+=\{\gamma\in \Gamma, \gamma\geq0 \mbox{ in } [0,1]\}$
$$J(u)= \inf_{\gamma \in \Gamma_+} \left(\max_{s\in [0,1]} J(\gamma(s))\right).$$
Thus $u\geq0$ and since $u$ satisfies  $\pfrac u \geq -||b||_\infty||u||_\infty^{r-p+1}u^{p-1}$ then Theorem 1.4 in \cite{hopf} implies that $u$ is positive in $\Omega$.
\end{proof}

\subsection{Convergence properties}\label{sec21}
In this section, we establish Theorem \ref{unifbu}. More precisely, subsection \ref{sup} gives the proof of {\it (i)}, the second point is obtained by Theorem \ref{conv0}. We also study the behavior of the sequence of solutions in the superlinear case obtained in subsection \ref{superl}. 

\subsubsection{As \texorpdfstring{$\lambda\to \sup \Lambda_q$}{lambda to sup Lambda q}}\label{sup}\ \\
We first study the uniform blow up on compacts subset of $\Omega_0$ as $\lambda \to \sup \Lambda_q$ using a method similar to that in \cite{lopezgomez}. For this, for the case $q=p-1$,  we introduce an auxiliary eigenvalue problem: let $\mu>0$,
\begin{equation*}\label{PbI}\tag{$\mathcal Q_\mu$}
\left\{\begin{array}{l l}
\pfrac \psi_\mu=\lambda_\mu \psi_\mu^{p-1}  - \mu b(x)\psi_\mu^{p-1} &\text{in} \;\Omega,\\
\psi_\mu=0 &\text{in} \; ( \Rn \backslash \Omega), \\
\psi_\mu>0 &\text{in} \; \Omega .
\end{array}\right.
\end{equation*}
We first have:
\begin{theorem}\label{thvp}
For any $\mu>0$, there exist $\lambda_\mu\in (0,\lambda_1(\Omega_0))$ and a positive function 
$\psi_\mu\in \Wspz \cap L^\infty(\Omega)$ satisfying $\|\psi_\mu\|_{L^p(\Omega)}=1$,  such that
\begin{equation*}
    \begin{split}
        \lambda_\mu&=\inf_{w\in \Wspz, \|w\|_{L^p(\Omega)}=1} \|w\|_{\Wspz}^p+\mu\int_\Omega b|w|^p\,dx\\
        &= \|\psi_\mu\|_{\Wspz}^p+\mu\int_\Omega b\psi_\mu ^p\,dx.   
    \end{split}
\end{equation*}
Moreover, the sequence $(\psi_\mu,\lambda_\mu)_\mu$ converges to $(\vpz,\lambda_1(\Omega_0))$ in $C(\overline \Omega)\times \R$ as $\mu\to +\infty$ where $\vpz$ is the normalized eigenfunction of $\pfrac$ associated to $\lambda_1(\Omega_0)$.
\end{theorem}

\begin{proof} The existence of $(\psi_\mu,\lambda_\mu)_\mu$ is given by Theorem 3.4 in \cite{refeigen}. Moreover Remark 3.5 and  Lemma 3.2 of \cite{refeigen} imply that $\Psi_\mu >0$ on $\Omega$ and  Theorem 3.2 of \cite{franzina} yields that $\psi_\mu\in L^\infty(\Omega)$ for any $\mu>0$.\\
By definition of $\lambda_\mu$, the sequence $(\lambda_\mu)_\mu$ is nonincreasing and for any $w$ belonging to $\Wspzz\subset \Wspz$ such that $\|w\|_{L^p(\Omega_0)}=1$, we have 
$$\|\psi_\mu\|_{\Wspz}^p+\mu\int_\Omega b\psi_\mu^p\,dx=\lambda_\mu\leq ||w||^p_\Wspz= ||w||^p_\Wspzz.$$
In particular, taking the eigenfunction associated to $\lambda_1(\Omega_0)$, we obtain that $\lambda_\mu\leq\lambda_1(\Omega_0)$. We also deduce that the sequence $(\psi_\mu)_\mu $ is bounded in $\Wspz$ and
\begin{equation}\label{unp}
\int_{\Omega}b\psi_\mu^p\,dx\to 0 \text{ as } \mu \to \infty.
\end{equation}
Then, up to a subsequence, there exist $\psi\in \Wspz$ and $\lambda_\infty\in (0,\lambda_1(\Omega_0)]$ such that $\psi_\mu$ converges weakly to $\psi$ in $\Wspz$. Thus, by a Sobolev compact embedding, $\psi_\mu$ converges to $\psi$ in $L^p(\Omega)$, {\it a.e.} in $\Omega$ and $\lambda_\mu \to \lambda_\infty$ as $\mu \to +\infty$.\\
Furthermore \eqref{unp} implies that $\int_\Omega b\psi^p\,dx=0$ thus $\psi=0 $ on $\Omega_0^c$ and $\psi$ belongs to $\Wspzz$.\\ 
For any $\mu>0$, $\pfrac \psi_\mu \leq  \lambda_1(\Omega_0) \psi_\mu^{p-1}$ in $\Omega$ and noting that  Theorem 3.2 in \cite{franzina} is still valid for positive subsolutions as $\psi_\mu$, this yields that $(\psi_\mu)_\mu$ is uniformly bounded in $L^\infty(\Omega)$.\\
From \cite[Theorem 2.7]{hold}, we deduce that $(\psi_\mu)_\mu$ is uniformly bounded in $C^s(\Rn)$ and hence $(\psi_\mu)_\mu$ is equicontinuous in $C(K)$ for any $K$ compact set of $\Rn$ and thus  Ascoli-Arzela Theorem insures, up to a subsequence, that $\psi_\mu$ converges to $\psi$ in $C(\overline \Omega)$.\\
Since $(\psi_\mu)_\mu$ is bounded in $\Wspz$, the sequence $\big(\frac{\pow{\psi_\mu(x)-\psi_\mu(y)}{p-1}}{|x-y|^{(d+sp)/p'}}\big)_\mu$ is bounded in $L^{p'}(\Rn\times\Rn)$.
By compactness arguments, 
we identify
$$\frac{\pow{\psi_\mu(x)-\psi_\mu(y)}{p-1}}{|x-y|^{(d+sp)/p'}} \rightharpoonup \frac{\pow{\psi(x)-\psi(y)}{p-1}}{|x-y|^{(d+sp)/p'}} \mbox{  in } L^{p'}(\Rn\times\Rn).$$ 
We infer that, taking $\phi \in \Wspzz$, 
$$\langle \pfrac \psi, \phi\rangle=\lim_{\mu \to +\infty}\langle \pfrac \psi_\mu, \phi\rangle=\lim_{\mu \to +\infty}\lambda_\mu \int_{\Omega_0}\psi_\mu^{p-1}\phi=\lambda_\infty \int_{\Omega_0}\psi^{p-1}\phi.$$
Then, Theorem 3.7 in \cite{refeigen} and $||\psi_\mu||_{L^p(\Omega)}=1$ imply that $\lambda_\infty= \lambda_1(\Omega_0)$ and $\psi=\vpz$.\\ 
We now establish that $\lambda_\mu<\lambda_1(\Omega_0)$. We argue by contradiction assuming that $\lambda_\mu=\lambda_1(\Omega_0)$ for some $\mu$. Then, $\vpz$ satisfies:
\begin{equation*}
\|\vpz\|_{\Wspz}^p+\mu\int_\Omega b|\vpz|^p\,dx=\lambda_1(\Omega_0)=\lambda_\mu.
\end{equation*}
Hence $\vpz$ is a nonnegative solution of \eqref{PbI} and thus Theorem 2.9 of \cite{refeigen} implies $\vpz>0$ {\it a.e.} in $\Omega$ which contradicts $\vpz\in \Wspzz$.
\end{proof}
By a comparison principle, we prove the following result of uniform blow up on the compact subsets of $\Omega_0$.
\begin{theorem}\label{Thexplo0}
The sequence $(u_\lambda)_\lambda$ blows up uniformly on the compact subsets of $\Omega_0$ as $\lambda \to \sup \Lambda_q$.
\end{theorem}

\begin{proof}
Consider first the case $q=p-1$, then $\sup \Lambda_q=\lambda_1(\Omega_0)$.\\
Let $\mu>0$, taking $\psi_\mu$ and $\lambda_\mu$ defined by Theorem \ref{thvp} and normalizing $\psi_\mu$  such that $||\psi_\mu||_{L^\infty(\Omega)}=1$, we define $\underline{u}_\lambda=\mu ^\frac{1}{r-p+1}\psi_\mu$ which satisfies
$$\pfrac \underline{u}_\lambda = \lambda_\mu \underline{u}_\lambda^{p-1}-b\underline{u}_\lambda ^r\psi_\mu^{p-1-r}.$$
Since $r>p-1$ and $||\psi_\mu||_\infty=1$, for any $\lambda$ close enough to $\lambda_1(\Omega_0)$ ({\it i.e.} $\lambda_\mu\leq \lambda<\lambda_1(\Omega_0)$), we get:
\begin{equation}\label{2}
\pfrac \underline{u}_\lambda\leq \lambda \underline{u}_\lambda^{p-1}-b\underline{u}_\lambda^r
\end{equation}
hence $\underline{u}_\lambda$ is a subsolution of \eqref{PbLE}. Thus,  $u_\lambda \geq \underline{u}_\lambda=\mu^\frac{1}{r-p+1} \psi_\mu$  {\it a.e.} in  $\Omega$  for $\lambda\in [\lambda_\mu,\lambda_1(\Omega_0))$ by Proposition \ref{mono}.
Hence taking $\mu $ goes to $+\infty$, Theorem \ref{thvp} insures the conclusion of the proof in the case $q=p-1$.\\
Consider now the case $q<p-1$, then $\sup \Lambda_q=+\infty$.
Let $V$ be the unique solution of the following problem
\begin{equation*}
\left\{\begin{array}{l l}
\pfrac V=V^q &\text{in} \;\Omega_0,\\
V=0 &\text{in} \; ( \Rn \backslash \Omega_0), \\
V>0 &\text{in } \; \Omega_0.
\end{array}\right.
\end{equation*}
The solution $V$ is obtained by minimization and satisfies $V\geq  Cd(\cdot,\Omega_0^c)^s$ for some constant $C>0$  (see for instance Theorem 4.3 in \cite{MR4913772}).\\
Define, for any $\lambda \geq 1$, $v_\lambda=\lambda^{\frac{1}{p-1-q}}V$.

Since the solution $u_\lambda$ of \eqref{PbLE} is a supersolution of  $\pfrac w=\lambda w^q$ in $\Omega_0$, Remark \ref{monob} implies that  $u_\lambda\geq v_\lambda$ {\it a.e.} in $\Omega_0$.
Hence $u_\lambda\geq v_\lambda\geq C\lambda^\frac{1}{p-1-q}d(\cdot,\Omega_0^c)^s$ and we conclude the blow-up of $u_\lambda$ on the compact subsets of $\Omega_0$.
\end{proof}

\begin{theorem}\label{Thexplo}
The sequence $(u_\lambda)_\lambda$ blows up uniformly on the compact subsets of $\Omega$ as $\lambda \to \sup \Lambda_q$ under the additional condition $p\geq 2$ if $q<p-1$.
\end{theorem}
\begin{proof}
Let $K$ be a compact subset of $\Omega$.
Let $R<\frac12 d(K,\partial \Omega)$
and we consider the finite family $(x_k)\subset K$ such that $K\subset \bigcup B(x_k,R)$.
For any $k$, we set $\mathcal O= B(x_k,2R)\subset \Omega$ and we consider two cases $\overline{\mathcal O}\subset \Omega_0$ and  $\overline{\mathcal O}\not\subset \Omega_0$. In the first case, we apply Theorem \ref{Thexplo0} in $B(x_k,R)$.\\
For the second case, there exists $\omega \subset \Omega_0$ verifying $\overline \omega \subset \Omega_0$ and $d(\omega,\overline{\mathcal O})>0$.

We define $v_\lambda = u_\lambda- u_\lambda\mathbb{1}_\omega$ where $\mathbb{1}_\omega(x)=1$ if $x\in \omega$ and  $\mathbb{1}_\omega(x)=0$ otherwise, then by \cite[Lemma 2.8] {globalhold}, $v_\lambda$ satisfies in the weak sense on $W^{s,p}_0(\mathcal O)$
\begin{equation}\label{eq8}
\pfrac v_\lambda+bv_\lambda^r = \lambda u_\lambda^{q}+h_\lambda
\end{equation}
where, recalling the notation $\pow{t}{p-1}=|t|^{p-2}t$, for any $x\in \mathcal O$:
$$h_\lambda(x)=2\int_\omega \frac{\pow{u_\lambda(x)}{p-1}-\pow{u_\lambda(x)-u_\lambda(y)}{p-1}}{|x-y|^{d+sp}}dy.$$

First, for $q=p-1$, we choose $\alpha>0$ a suitable parameter independent of $\lambda$ small enough such that  for any $x\in \mathcal O$, we have 
\begin{equation}
 \lambda u_\lambda^{p-1 }(x)+h_\lambda(x)  \geq 2\int_\omega \frac{(1+\alpha)\pow{u_\lambda(x)}{p-1}-\pow{u_\lambda(x)-u_\lambda(y)}{p-1}}{|x-y|^{d+sp}}dy.
\end{equation}
Noting that, for any  $X\geq 0$, $Y>0$, $(1+\alpha)X^{p-1}-\pow{X-Y}{p-1}\geq  c_p Y^{p-1}$
where the constant $c_p$ depends on $\alpha$ and $p$. Hence,  we obtain, for any $x\in \mathcal O$ 
\begin{equation}\label{eq9}
h_\lambda(x) + \lambda u_\lambda^{p-1 }(x)\geq 2c_p\int_\omega \frac{u^{p-1}_\lambda(y)}{|x-y|^{d+sp}}dy\geq  C(\inf_\omega u_\lambda)^{p-1}
\end{equation}
where $C$ is independent of $\lambda$.\\
Now consider the case $q<p-1$ where $p\geq 2$.  The inequality \eqref{ineg2} implies directly that for any $x\in \mathcal O$, $h_\lambda(x) \geq C (\inf_{ \omega}u_\lambda)^{p-1}$ where $C>0$ is independent of $\lambda$.\\
Hence, the inequality \eqref{eq9} holds for $q<p-1$ and $p\geq 2$. Thus by Theorem \ref{Thexplo0}, we deduce in any case that the right-hand side in \eqref{eq8} can be chosen large enough in $\mathcal O$ as $\lambda\to \sup \Lambda_q$.\\
Consider now $w\in \Wspz\cap L^\infty(\Omega)$ the positive solution of 
\begin{equation*}
\left\{
\begin{array}{l l}
\pfrac w = 1 &\text{in} \;\mathcal O,\\
w= 0 &\text{in} \; ( \Rn \backslash \mathcal O),
\end{array}\right.
\end{equation*}
satisfying $w\geq cd(\cdot,\mathcal{O}^c)^s$ for some constant $c>0$ (see \cite[Theorem 1.5]{hopf}).\\
Then, for $\rho>0$ we set $w_\rho=\rho w$ which satisfies for $\lambda$ close enough to $\sup \Lambda_q$, for any $\varphi \in W^{s,p}_0(\mathcal O)$, $\varphi\geq 0$
\begin{equation*}
\begin{split}
\langle \pfrac w_\rho,\varphi\rangle + \int_{\mathcal O} bw_\rho^r\varphi\,dx 
&\leq \langle \pfrac v_\lambda,\varphi\rangle +\int_{\mathcal O} bv_\lambda^r \varphi\,dx.
\end{split}
\end{equation*} 
Taking $\varphi=(w_\rho-v_\lambda)_+$, we deduce that 
$u_\lambda =v_\lambda\geq w_\rho\geq c\rho d(\cdot,\mathcal{O}^c)^s \mbox{ in }\mathcal{O}$
where $c>0$ is independent of $\rho$ and we infer that for any $k$, for any $\rho>0$ and for $\lambda$ close enough to $\sup \Lambda_q$, 
$$\inf_{B(x_k,R)} u_\lambda\geq \tilde C_k \rho.$$   
Finally, we conclude that $u_\lambda$ blows up in any ball $B(x_k,R)$ and hence blows up uniformly on the compact set $K$.
\end{proof}
\subsubsection{As \texorpdfstring{$\lambda\to \inf \Lambda_q$}{lambda to inf Lambda q}}
\begin{theorem}\label{conv0}
The sequence $(u_\lambda)_\lambda$ converges uniformly to $0$ as $\lambda \to \inf \Lambda_q$.
\end{theorem}
\begin{proof}
From Proposition \ref{mono}, for $\lambda$ close to $\inf \Lambda_q$, we have that the sequence $(u_\lambda)_\lambda$ is bounded in $L^\infty(\Omega)$.\\
We deduce taking the test function $u_\lambda$ in \eqref{ws}, that  the sequence $(u_\lambda)_\lambda$  is also bounded in $\Wspz$. Then, up to a subsequence and for some $u\in \Wspz$, $u_\lambda \rightharpoonup u$ in $\Wspz$.   By Sobolev embedding and interpolation, we get the convergence in $L^m(\Omega)$ for any $m\geq 1$.
Moreover Theorem 1.1 in \cite{globalhold} and Ascoli-Arzela Theorem give $u_\lambda\to u$ in $C(\overline\Omega)$.\\
Proceeding as in the proof of Theorem \ref{thvp}, we establish that $u$ is a weak solution of 
$$\left\{
\begin{array}{l l}
\pfrac u + bu^r = (\inf \Lambda_q )u^{q} & \mbox{ in } \Omega\\
u=0 &\mbox { in } \Rn\backslash \Omega.
\end{array}\right.
$$
In the case $q<p-1$, {\it i.e.} $\inf \Lambda_q=0$, $w=0$ is the unique solution by Proposition 2.1 in \cite{MR4913772}. For $q=p-1$, {\it i.e.} $\inf \Lambda_q=\lambda_1(\Omega)$,  Theorem \ref{thexi1} implies $u=0$.
\end{proof}
\subsubsection{For \texorpdfstring{$q>p-1$}{q>p-1}}
\ \\
In this subsection, we get the asymptotic behavior for the superlinear case.\\
We set $\Sigma_q\subseteq \R^+_*$ the set of $\lambda$ such that \eqref{PbLE} admits a weak solution. We denote that for $q\leq p-1$,  $\Sigma_q=\Lambda_q$ by Theorem \ref{thexi1} and for $r$ and $q$ as in Theorem \ref{thexi2} we have $\Sigma_q=(0,+\infty)$.
\begin{theorem}\label{q>p-1}
Let $q>p-1$. Assume that $\inf \Sigma_q=0$, then the sequence $(||u_\lambda||_\infty)_{\lambda\in \Sigma_q} $ goes to $\infty$ as $\lambda \to 0$.
\end{theorem}
\begin{proof}
Let $(\lambda_n)$ be a minimizing sequence of $\Sigma_q$. 
It suffices to show that any subsequences of $(u_{n}):=(u_{\lambda_n})$ are unbounded. For that, we argue by contradiction assuming that there exists $(\lambda_{n'})$ such that $||u_{n'}||_\infty\leq M$ for any $n'$. Taking $u_{n'}$ as test function in \eqref{ws}, this yields 
\begin{align*}
||u_{n'}||^p_\Wspz& \leq \lambda_{n'} \int_\Omega u_{n'}^q\,dx \leq \lambda_{n'} M^{q+1-p} ||u_{n'}||^{p}_{L^p(\Omega)}.
\end{align*}
Hence, for $n'$ large enough we obtain a contradiction with the Poincaré inequality  since $u_{n'}\not\equiv0$.
\end{proof}

\section{The parabolic problem}\label{sec3}

\subsection{Existence of solutions and properties}\label{sec31}
In this section, we prove Theorem \ref{mainpara}. We start by proving a local existence result.

\begin{theorem}\label{exipara}
Let $u_0\in \Wspz\cap L^\infty(\Omega)$ be a nonnegative function and let $q>0$. Then, there exists $T_{*}\in (0,+\infty]$ such that, for any $T<T_*$, \eqref{pbp} admits a weak solution on $Q_T$ belonging to $C([0,T];L^m(\Omega))$ for any $m\in [1,+\infty)$. 
\end{theorem}
\begin{proof}
The proof follows the ones of Theorem 3.1 and Corollary 3.2 in \cite{MR4913772} adapting some points. For the convenience of the readers, we only give the idea of the proof and the important points which are necessary for the other proofs. First, we replace the operator $\pfrac$ by the operator $\mathcal A$ which satisfies the maximum principle by Remark \ref{pmbis} and we define, for some $R>0$, the truncated problem
\begin{equation*}\label{pbpR}\tag{$\mathcal P_{\lambda,R}$}
   \left\{ \begin{array}{l l}
     \partial_t u+ \mathcal A u= \lambda \min(R,u)^q &\text{in} \;Q_T,\\
      u=0 &\text{in} \;(0,T)\times\Omega^c,\\
      u\geq 0 &\text{in} \;Q_T,\\
      u(0)=u_0 &\text{in } \Omega.
    \end{array}\right.
\end{equation*}
We also consider the associated discretization scheme of \eqref{pbpR}: 
\begin{equation}\label{approx}
\left\{\begin{array}{l l}
\dfrac{u_n-u_{n-1}}{\dt}+\mathcal A u_n=\lambda \min(R,u_{n-1})^q & \mbox{ in } \Omega,\\
u_n=0 & \mbox{ in } \Rn\backslash \Omega,\\
u_n\geq 0 & \mbox{ in } \Omega
\end{array}\right.
\end{equation}
where $\dt=\frac{T}{N}$ for some $N\in \mathbb N^*$. The sequence $(u_n)$ is well defined and bounded in $L^\infty(\Omega)$, more precisely, we have {for $1\leq n\leq N$}
\begin{equation}\label{pas9}
\|u_n\|_{L^\infty(\Omega)}\leq \|u_0\|_{L^\infty(\Omega)}+\lambda T R^q.
\end{equation}
Then, we define the functions $u_\dt =u_n$  and $\tilde u _\dt=\frac{\cdot-t_{n-1}}{\dt}(u_n-u_{n-1})+u_{n-1}$ defined on $ [(n-1)\dt,n\dt)$ for $1\leq n\leq N$ which satisfy
\begin{equation*}\label{pbdt}
\partial_t \tilde u_\dt+\mathcal A u_\dt=\lambda u_{\dt}^q(\cdot-\dt) \text{ in } Q_T.
\end{equation*}
As in the proof of Theorem 3.1 in \cite{MR4913772}, we obtain estimates of the sequences $(u_\dt)_{\dt}$ and $(\tilde u_\dt)_{\dt}$ which imply by compactness arguments the convergence in suitable spaces. Hence, for any $T>0$ and for any $R$, we get a weak solution (in the sense of Definition \ref{defsol2}) of the problem \eqref{pbpR} $u_R\in C([0,T];L^m(\Omega))$, for any $m\in [1,+\infty)$   
 limits of $(u_\dt)_{\dt}$ and $(\tilde u_\dt)_{\dt}$ and satisfying by \eqref{pas9}
\begin{equation}\label{9}
\|u_R\|_{L^\infty(Q_T)}\leq \|u_0\|_{L^\infty(\Omega)}+\lambda T R^q.
\end{equation}
Finally, as in the proof of Corollary 3.2 of \cite{MR4913772}, we deduce the existence of $T_*=T_*(q,u_0,\lambda)$  such that for suitable $R>0$, $\|u_R\|_{L^\infty(Q_T)}<R$ for any $T<T_*$ and hence $u_R$ is a weak solution of \eqref{pbp} on $Q_T$.
\end{proof}
\begin{remark}\label{parab}
As in Corollary 3.2 of \cite{MR4913772}, $T_*$ does not depend on $u_0$ if $q\leq 1$ more precisely we have that $T_*=+\infty$ when $q<1$ and $T_*=\frac1\lambda$ when $q=1$.
\end{remark}
Under additional conditions, we establish a existence result for any time.
\begin{theorem}\label{para1}
Let $u_0\in \Wspz\cap L^\infty(\Omega)$ be a nonnegative function. Assume $q\in(0, p-1]$ and $\lambda< \sup\Lambda_q$.  Then, for any $T>0$, \eqref{pbp} admits a weak solution on $Q_T$ belonging to $C([0,T];L^m(\Omega))$ for any $m\in [1,+\infty)$. 
\end{theorem}

\begin{proof}
The proof is a close adaptation of the one of Theorem 1.4 in \cite{MR4913772}. For this reason, we skip it however we only mention the main changes.\\
First,  we establish the existence of a weak solution $u_R$ of \eqref{pbpR}. 
As in the proof of Theorem 1.4 in \cite{MR4913772}, it is sufficient to prove that $u_R$ is uniformly bounded independently of $T$ and $R$ in $Q_T$. For that, it suffices to obtain a uniform bound of the sequence $(u_n)$ defined by \eqref{approx}.\\
Let $\tilde{\lambda}\in \Lambda_q$ such that $\tilde \lambda>\lambda$ and let $\tilde u$ be the solution of $(\mathcal E_{\tilde \lambda})$ in  $\widetilde{\mathcal  O}  \supset \overline \Omega$ where the function $b$ is extended by a positive function $\tilde b$ on $\widetilde{\mathcal  O} \backslash \Omega$ such that $\tilde b^{-1}(\{0\})=\overline \Omega_0$. The function  $\tilde u\in W^{s,p}_0(\widetilde{\mathcal  O})$ is well defined by Theorem \ref{thexi1} because $\sup \Lambda_q$ only depends on $\Omega_0$.\\
Then, the function $\overline{u}=\alpha \tilde{u}$ with $\alpha\geq 1$ is a supersolution of \eqref{PbLE} in $\Omega$ by Proposition \ref{mono} and since  $q\leq p-1< r$. \\
Moreover, by Theorem \ref{thexi1}, $\tilde{u}\geq C$ in $\overline\Omega$ hence for $\alpha$ large enough  we have $\overline{u}\geq u_0$ {\it a.e.} in $\Omega$ and we choose $R\geq\|\overline u\|_{L^\infty(\Omega)}$. \\
Hence by induction argument, we have for any $\varphi\in \Wspz$, $\varphi \geq 0$.
\begin{equation*}
\begin{split}
\int_\Omega \frac{u_n-\overline{u}}{\dt} \varphi \,dxdt+\langle \mathcal Au_n-\mathcal A\overline u,\varphi\rangle=\int_\Omega &\frac{u_{n-1}-\overline{u}}{\dt}\varphi \,dxdt\\
&+\lambda \int_\Omega (u_{n-1}^q-\overline{u}^q)\varphi \,dxdt\leq 0,
\end{split}
\end{equation*}
and we deduce that for any $n$ 
\begin{equation}\label{ot9}
0\leq u_n\leq \overline u \mbox{  and  then } \|u_R\|_{L^\infty(Q_T)}\leq \|\overline u\|_{L^\infty(\Omega)}.
\end{equation} 
Finally, by the choice of $R$, we deduce that $u_R$ is a weak solution of \ref{pbp} in $Q_T$ for any $T>0$.
\end{proof}

\begin{remark}\label{compari}
For any $q>0$, we have a comparison principle about the solutions obtained by Theorem \ref{exipara} and \ref{para1}. More precisely, consider $\Omega \subseteq\widetilde \Omega$, $u_0\in \Wspz\cap L^\infty(\Omega)$ and $\tilde u_0\in W^{s,p}_0(\widetilde \Omega)\cap L^\infty(\widetilde \Omega)$ such that $0\leq u_0\leq \tilde u_0$, then the solutions $u$, $\tilde u$ obtained by Theorem \ref{exipara} and \ref{para1} are ordered {\it i.e.} $u\leq \tilde u$.  Indeed the comparison principle arises from the discretization scheme, Remark \ref{pmbis} and an induction argument.
\end{remark}

\begin{proposition}\label{para2}
Let $u_0\in \Wspz\cap L^\infty(\Omega)$ be a nonnegative function. Assume either $q\in(0, p-1]$ and $\lambda< \sup\Lambda_q$ or $q\leq 1$.  Then, there exists $u_{glob}\in L^\infty_{loc}((0,+\infty),L^\infty(\Omega))$ a global solution of \eqref{pbp} {\it i.e.} $u_{glob}$ is a weak solution of  \eqref{pbp} on $Q_T$ for any $T>0$. 
Moreover $u_{glob}$ belongs to $C([0,\infty);L^m(\Omega))$ for any $m\in [1,+\infty)$.
\end{proposition}

\begin{proof}
Let $\widetilde T\in (0,\frac{1}{\lambda})$ be fixed and independent of $u_0$. By Theorem \ref{exipara}, \ref{para1} and Remark \ref{parab}, for any $q\leq 1$ or $q\leq p-1$ and $\lambda< \sup \Lambda_q$, there exists a weak solution $u_1$ of \eqref{pbp} on $Q_{\widetilde T}$. By Definition \ref{defsol2}, $u_1(\widetilde T,\cdot)$ belongs to $L^\infty(\Omega)\cap \Wspz$ hence we apply either Theorem \ref{exipara} or Theorem \ref{para1} with the initial condition $u_1(\widetilde T,\cdot)$ and we get a weak solution $u_2$ of \eqref{pbp} on $Q_{\widetilde T}$.\\
Then we proceed by induction argument to construct a sequence $(u_n)\subset X_{\widetilde T}$ such that, for any $n\in \mathbb N^*$, $u_n$ is a weak solution of \eqref{pbp} on $Q_{\widetilde T}$ with the initial data $u_n(0)=u_{n-1}(\widetilde T)$.\\
We conclude defining $u_{glob}$ as follows $u_{glob}= u_{n+1}(\cdot-n\widetilde T,\cdot)$ on the interval $[n\widetilde T, (n+1)\widetilde T]\times \Omega$.    
\end{proof}
\begin{remark}
For $q\in(0, p-1]$ and $\lambda< \sup\Lambda_q$, $u_{glob}\in L^\infty([0,\infty)\times \Omega)$. 
\end{remark}
Now we deal with the question of the uniqueness establishing the following result:
\begin{proposition}\label{compara}
Let $u_0\in \Wspz\cap L^\infty(\Omega)$ be a nonnegative function. Let $T>0$ and let $q\geq 1$. Then, \eqref{pbp} admits at most a weak solution in the sense of Definition \ref{defsol2}.
\end{proposition}
\begin{proof}
Let $u$, $v$ be two solutions of \eqref{pbp} on $Q_T$. Taking $\phi=u-v$ in \eqref{fv2},  we obtain  for any $t\leq T$:
$$\frac12\|(u-v)(t)\|^2_{L^2(\Omega)}\leq \int_0^t\int_\Omega\lambda(u^q-v^q)(u-v)dxd\tau.$$
Since  $u$, $v\in L^\infty(Q_T)$ and the mapping $y\mapsto y^q$ is locally Lipschitz in $\mathbb R$, the previous inequality yields
$$\frac12\|(u-v)(t)\|^2_{L^2(\Omega)}\leq C\int_0^t\int_\Omega\|(u-v)(\tau)\|^2_{L^2(\Omega)}d\tau.$$
Finally, we conclude applying the Gronwall Lemma. 
\end{proof}
We introduce the energy functional of \eqref{pbp} given, for any $v\in \Wspz$, by:
$$E(v)=\frac{1}{p}\|v\|^p_\Wspz-\frac{\lambda}{q+1}\|v\|^{q+1}_{L^
{q+1}(\Omega)}+\frac{1}{r+1}\int_\Omega bv^{r+1}\,dx.$$
Then, we obtain
\begin{proposition}\label{propReg}
Let $u$ be the weak solution of \eqref{pbp} obtained by Theorem \ref{exipara} or \ref{para1} in $Q_T$ for some suitable $T>0$ with the initial data $u_0\in \Wspz\cap L^\infty(\Omega)$, $u_0\geq 0$ {\it a.e.} in $\Omega$. Then, $u\in C([0,T);\Wspz)$ and
\begin{equation}\label{energyineq}
\int_{t_0}^{t_1} \int_\Omega |\partial_t u|^2dxdt+E(u(t_1)) = E(u(t_0)),
\end{equation}
for any $t_0$, $t_1\in [0,T)$ such that $t_1 \geq t_0$. Additionally, if $q\geq 1$ and $u_0\in D(\mathcal A)$ then  $u$ belongs to $C([0,T];C(\overline \Omega))$.
\end{proposition}
\begin{proof}
Let $t_0\in [0,T)$, Proposition 3.4 of \cite{MR4913772} implies that for any $t\in (t_0,T)$ 
\begin{equation}\label{energyineq1}
\begin{split}
\int_{t_0}^{t} \int_\Omega  |\partial_t u|^2\,dxdt+&\frac{1}{p}\|u(t)\|^p_\Wspz \\
&\leq \frac{1}{p}\|u(t_0)\|^p_\Wspz +\int_{t_0}^{t}\int_\Omega (\lambda u^q-bu^r) \partial_t u\,dxdt,
\end{split}
\end{equation}
thus
\begin{equation*}\label{energyineq2}
\int_{t_0}^{t} \int_\Omega |\partial_t u|^2dxdt+E(u(t)) \leq E(u(t_0)).
\end{equation*}
Hence passing to the limit in the previous inequality as $t_k\to t_0$ with for any $k$, $t_k>t_0$, we deduce that $\limsup_{k\to\infty}\|u(t_k)\|_\Wspz\leq \|u(t_0)\|_\Wspz$. On the other hand, $u\in L^\infty([0,T];\Wspz)\cap C([0,T];L^m(\Omega))$ for any $m\in [1,+\infty)$ thus the mapping $t\mapsto u(t)$ is weakly continuous from $[0,T]$ to $\Wspz$. We finally get that the right continuity of $u$ in $t_0$, for any $t_0\in [0,T)$.\\
Let $h>0$ be small enough. Choosing  $\tau_h u=\frac{u(\cdot+h,\cdot)-u}{h}$ as test function in \eqref{fv2} first on $[0,t_0]$ and then on $[0,t]$, this yields subtracting the both equations and by convexity of $\frac1p\|\cdot\|^p_\Wspz$:
\begin{equation*}
\begin{split}
\int_{t_0}^{t}\int_\Omega\partial_tu \tau_h u\;dxdt&+\frac{1}{hp}\int_{t}^{t+h}\|u\|^p_\Wspz dt\\
&\geq \frac{1}{hp}\int_{t_0}^{t_0+h}\|u\|^p_\Wspz dt+\int_{t_0}^{t}\int_\Omega(\lambda u^q-bu^r) \tau_h u\;dxdt.
\end{split}
\end{equation*}
Since $\partial_t u\in L^2(Q_T)$, $u\in L^\infty(Q_T)\cap L^\infty([0,T];\Wspz)$ and by the Dominated Convergence Theorem, we pass to the limit in the previous inequality as $h\to 0$. Hence we have for any $t\in (t_0,T)$
\begin{equation*}
\begin{split}
\int_{t_0}^{t}\int_\Omega|\partial_tu|^2dxdt&+\frac1p\|u(t)\|^p_\Wspz\\
&\geq \frac1p\|u(t_0)\|^p_\Wspz+\int_{t_0}^{t}\int_\Omega(\lambda u^q-bu^r) \partial_t u \;dxdt,
\end{split}
\end{equation*}
which gives the equality in \eqref{energyineq1} and hence \eqref{energyineq} holds. \\
Then, passing to the limit in \eqref{energyineq}, for $t_k\to t$ with $t_k\leq t$ we deduce that $\limsup_{k\to\infty}\|u(t_k)\|_\Wspz= \|u(t)\|_\Wspz$. Plugging with the weak continuity of $u$ in $t$, we establish the left continuity of $u$ in $t$, for any $t\in (0,T)$. Finally we conclude that $u\in C([0,T),\Wspz).$\\
We now take $q\geq 1$ and $u_0\in {D(\mathcal A)}$ and we show that $u\in C([0,T];C(\overline \Omega))$.\\
For that, we subtract the discretization scheme \eqref{approx} at the range $n$ from the one in $n-1$ then Proposition \ref{accre} implies
$$\|u_n-u_{n-1}\|_{L^\infty(\Omega)}\leq (1+C\dt\lambda)\|u_{n-1}-u_{n-2}\|_{L^\infty(\Omega)},$$
where  the constant $C$ is independent of $n$. Using discrete Gronwall  Lemma, we get for any $n\in \mathbb N^*$
$$\|u_n-u_{n-1}\|_{L^\infty(\Omega)}\leq e^{C\lambda T}\|u_1-u_0\|_{L^\infty(\Omega)}.$$
Since $u_0\in D(\mathcal A)$, we apply again Proposition \ref{accre} and we obtain
\begin{equation}\label{ps9}
\|u_n-u_{n-1}\|_{L^\infty(\Omega)}\leq \dt e^{C\lambda T} \|\lambda u_0^q-\mathcal Au_0\|_{L^\infty(\Omega)}=C(u_0,T)\dt.
\end{equation}
Hence, plugging \eqref{pas9},\eqref{ot9} and \eqref{ps9}, we deduce that there exists a constant $C$ such that for any $n\in \mathbb N^*$
$$\|\lambda u_{n-1}^q+bu_n^r -\frac{u_n-u_{n-1}}{\dt}\|_{L^\infty(\Omega)}\leq C,$$
and hence Theorem 2.7 in \cite{hold} implies that the sequence $(u_n)_n$ is bounded in $C^s(\Rn)$.\\
The Ascoli Theorem  implies up to a subsequence, for any $t\in [0,T]$, $\tilde u_\dt(t)$ converges to $u(t)$ in $L^\infty(\Omega)$ and $\tilde u_\dt(t)$ converges $u(t)$ in $C(\overline \Omega)$. \\
Since \eqref{ps9} holds for any $n$, $(\partial_t\tilde{u}_\dt)_\dt$ is uniformly bounded in $L^\infty(Q_T)$ hence there exists a constant $C>0$ independent of $\dt$ such that for any $t$, $t_0\in[0,T]$
$$\|\tilde u _\dt (t)-\tilde u_\dt (s) \|_{L^\infty(\Omega)}\leq C|t-s|,$$
thus passing to the limit as $\dt \to 0$,  we conclude that $u\in C([0,T];C(\overline \Omega))$.
\end{proof}
\begin{remark}    
If $p\geq2$, $q\in[1,p-1]$, $\lambda< \sup\Lambda_q$ and $u_0\in \overline{D(\mathcal A)}^{L^\infty}\cap \Wspz$, then using compactness arguments, $u\in C([0,T];C(\overline \Omega)).$
\end{remark}
Finally we establish Theorem \ref{mainpara} 

\begin{proof}[Proof of Theorem \ref{mainpara}]
The existence of a solution of \eqref{pbp} is given by Theorems \ref{exipara} and \ref{para1}, the regularity by Proposition \ref{propReg}.
The first point comes from Proposition \ref{para2} and the uniqueness from Proposition \ref{compara}.
\end{proof}

\subsection{Asymptotic behaviors} \label{sec32}
\subsubsection{Case \texorpdfstring{$q\leq p-1$}{q leq p-1}}\ \\
We begin by proving Theorem \ref{stab1}. The arguments are classical, we construct barrier functions which converge to the stationary solution of \eqref{pbp}.
\begin{proof}[Proof of Theorem \ref{stab1}] We split the proof in several steps.\\
\textbf{Step 1.} Barrier functions\\
Let $\mathcal{O}_n:=\{x\in \R^N\,|\, d(x,\Omega)<\frac1n\}$. We consider $\overline u_{0,n}\in W^{s,p}_0(\mathcal{O}_n)$ the supersolution of \eqref{PbLE} constructed as in the proof of Theorem \ref{para1} such that $\overline u_{0,n}\geq u_0$ in $\Omega$.\\ 
We consider now $w_\lambda\in \Wspz \cap C^s(\Rn)$  the solution of \eqref{PbLE} in $\tilde \Omega \subset \Omega$ satisfying  $w_\lambda \leq cd(x,^c\tilde \Omega)^s$, such a solution exists as we take $\lambda\in \Lambda_q$ and assuming in addition $\Omega_0\subset\tilde \Omega$ if $q=p-1$. So we define $\underline{u}_0=\epsilon w_\lambda$,  which is a subsolution of \eqref{PbLE} in $\tilde \Omega$ such that $\underline{u}_0\leq u_0$ {\it a.e.} in $\Omega$ for $\epsilon$ small enough. \\
Let $\overline u_n\in C([0,+\infty),L^m(\mathcal O_n))$ be the global solution of \eqref{pbp} on $[0,\infty)\times \mathcal{O}_n$ obtained by Proposition \ref{para2} with $\overline u_n(0)=\overline u_{0,n}$. In the same way, we consider $\underline{u}\in C([0,+\infty),L^m(\Omega))$ the global solution of \eqref{pbp} with $\underline u(0)=\underline u_0$.\\ 
By construction,  the function $t\mapsto \underline u(t)$ is nondecreasing. Indeed, in the discretization scheme \eqref{approx} with $R$ large enough, 
we have for any $\varphi\in W^{s,p}_0(\tilde\Omega)$, $\varphi\geq 0$:
\begin{equation*}
\int_\Omega\frac{\underline{u}_1-\underline{u}_0}{\dt}\varphi\,dx+\langle \mathcal A \underline{u}_1,\varphi\rangle=\int_\Omega \lambda\underline{u}_0^q \varphi\,dx \geq \langle \mathcal A \underline{u}_0,\varphi\rangle,
\end{equation*}
and using $\underline{u}_1\geq 0$ on $\Omega\backslash\tilde \Omega$ and applying Proposition 2.1 of \cite{MR4913772}, we conclude first that $\underline{u}_1\geq \underline{u}_0$ {\it a.e.} in $\Omega$. By induction, we have $\underline{u}_n\geq \underline{u}_{n-1}$ and hence the approximation functions $\underline u_{\dt}$ and $\underline{\tilde  u}_{\dt}$ are nondecreasing in time. Thus taking $\dt\to 0$, we have $\underline{u}$ nondecreasing. \\
In the same way, considering the discretization scheme \eqref{approx} in $\mathcal O_n$, we establish that, for any $n\in \mathbb N^*$, the mapping $t\mapsto \overline u_n(t)$ is nonincreasing. Moreover, we get $\underline{u}(t,x)\leq u(t,x)\leq \overline{u}_n(t,x)$ for any $(t,x)\in [0,+\infty)\times \Omega$ and for any $n$.\\
\textbf{Step 2.} Convergence in $L^m(\Omega)$  to $u_\lambda$ the solution of \eqref{PbLE}\\
Define $\underline{u}_\infty=\lim_{t\to\infty} \underline u(t)$ and $\overline{u}_{\infty,n}=\lim_{t\to\infty} \overline u_n(t)$, thus $\underline{u}$,  $\overline{u}_n$ converge $\underline{u}_\infty$, $\overline{u}_{\infty,n}$ in $L^m(\Omega)$ for any $m\in [1,+\infty)$ by Dominated Convergence Theorem as $t\to \infty$.\\
In the following, the open set $\mathcal O$ represents either $\mathcal O=\Omega$ or $\mathcal O=\mathcal O_n$.\\
We introduce $\{S(t,\mathcal O)\}_{t\geq0}$ defined on $W^{s,p}_0(\mathcal O)\cap L^\infty(\mathcal O)$ by $S(t,\mathcal O)v_0=v(t)$ where $v$ is the global solution of \eqref{pbp} obtained by Theorem \ref{para1} and Proposition \ref{para2}. By the construction of $v$ in the proof of Proposition \ref{para2}, we have $v\in C([0,\infty),L^1(\mathcal O))$ and for $k\in \mathbb{N}$: $S(t+k\tilde T,\mathcal O)v_0=S(t,\mathcal O)S(k\tilde T,\mathcal O)v_0$ hence if $v(t)$ converges, we have
$$\lim_{k\to \infty} S(t+k\tilde T,\mathcal O)v_0=S(t,\mathcal O)\lim_{k\to \infty} S(k\tilde T,\mathcal O)v_0$$
From this, we deduce  that $\underline{u}_\infty$ is a nonnegative stationary solution of \eqref{PbLE} and by uniqueness, $\underline{u}_\infty=u_\lambda$. Similarly, we obtain $\overline{u}_{\infty,n}$ is the solution of \eqref{PbLE} in $\mathcal{O}_n$.\\
It remains to establish  that $\overline{u}_{\infty,n}$ converges to $u_\lambda$ as $n\to \infty$.
Using Proposition \ref{mono}, the sequence $(\overline{u}_{\infty,n})_n$ is nonincreasing and hence bounded in $L^\infty(\mathcal{O}_1)$ and also in $W^{s,p}_0(\mathcal{O}_1)$ by the variational formulation. Thus we have for any $\phi \in \Wspz\subset W^{s,p}_0(\mathcal{O}_n)$ and for any $n$
$$\langle \pfrac \overline{u}_{\infty,n}, \phi\rangle=\int_\Omega (\lambda \overline{u}_{\infty,n}^q-b\overline{u}_{\infty,n}^r)\phi,$$
and as in the proof of Theorem \ref{thvp}, we pass to the limit in the previous equation and we obtain that $\overline{u}_{\infty,n}\to u_\lambda$ in $L^m(\Omega)$ for any $m\in [1,+\infty)$.\\
By \textbf{Step 1.},  we have for any $n$: 
$$u_\lambda \leq \lim_{t\to + \infty} u(t)\leq \overline{u}_{\infty,n} \ \mbox{{\it a.e.} in } \Omega $$
and hence as $n\to \infty$, we deduce the pointwise convergence of $u(t)$ to $u_\lambda$. Finally, by the $L^\infty(\Omega)$ bounds we conclude that $u$ tends to $u_\lambda$  as $t$ goes to $\infty$ in $L^m(\Omega)$ for any $m$.\\
\textbf{Step 3.} Convergence in $L^\infty(\Omega)$\\
We now take $q>1$ and $u_0\leq Cd(\cdot, \Omega^c)^s$. From the previous estimate of $u_0$, we can choose directly $\widetilde {\mathcal O}=\Omega$  in the proof of Theorem \ref{para1} to build the supersolution. Indeed, we choose $\alpha $ large enough such that $\overline u_0=\alpha u_{\tilde \lambda}$, where $\tilde \lambda \geq\lambda$, is a supersolution of \eqref{PbLE} and $\overline u_0\geq u_0$.\\
By construction, $\overline{u}_0$ and $\underline{u}_0$ belong to $D(\mathcal A)$ hence  we have $\underline{u}(t),\,\overline{u}(t)\in C(\overline \Omega)$ by Proposition \ref{propReg}. Finally Dini Theorem yields that $\underline{u}(t)$, $\overline{u}(t)$ converge to $u_\lambda$ in $L^\infty(\Omega)$.
\end{proof}
\begin{remark}
For the case $q=p-1$ and for any $\lambda\leq\lambda_1(\Omega)$, we obtain a similar result as in Theorem \ref{stab1}, but in this case the limit is $0$. For the proof, we take $0$ as a subsolution and the condition $u_0\geq cd(\cdot,\tilde \Omega^c)^s$ is not necessary.
\end{remark}
\begin{proof}[Proof of Theorem \ref{stabex}]
For $q\leq 1$, the solution $u$ is global and for $q>1$, Remark \ref{tmax} deals with $T_{max}<\infty$. Hence we just consider the case $t\to +\infty$.

We set $\underline{u}_0=\epsilon w$ where $w$ is the positive solution of:
\begin{equation*}
\left\{ \begin{array}{l l}
\pfrac w=\lambda_1(\tilde \Omega) w^{p-1} &\text{in} \;\tilde\Omega,\\
w=0 &\text{in} \; ( \Rn \backslash \tilde \Omega). \\
\end{array}\right.
\end{equation*}
As in the proof of Theorem \ref{stab1}, the solution $\underline u$ of \eqref{pbp} with $\underline u(0,.)=\underline u_0$ in $\Omega$, is nondecreasing in time and satisfies 
$\underline{u}\leq u$.\\
 Let $\ell= \lim_{t\to\infty}\|\underline u(t)\|_{L^\infty(\Omega)}$.  Assume that $\ell< \infty$ , then  $\|\underline u\|_{L^\infty(Q_\infty)}\leq C$ where $Q_\infty=(0,+\infty)\times \Omega$. Thus, as in the proof of Theorem \ref{stab1}, $\underline u$ converges to a stationary solution of \eqref{pbp} for $\lambda\geq \lambda_1(\tilde \Omega)\geq \lambda_1(\Omega_0)$ which contradicts Theorem \ref{thexi1}. Hence $\ell=+\infty$ and we conclude the proof of Theorem \ref{stabex} by comparison.
\end{proof}
\subsubsection{Case  \texorpdfstring{$q>p-1$}{q>p-1}}\ \\
We consider the following problem for $\lambda>0$
\begin{equation*}\label{8}\tag{$\mathbb P_0$}
    \left\{ \begin{array}{l l}
     \partial_t v+ \pfrac v= \lambda v^{q} &\text{in} \;(0,T)\times\Omega_0 ,\\
      v=0 &\text{in} \; (0,T)\times\Omega_0^c,\\
      v\geq 0 &\text{in} \; (0,T)\times\Omega_0,\\
      v(0)=v_0 &\text{in } \Omega_0,
    \end{array}\right.
\end{equation*}
where $v_0\in \Wspzz\cap L^\infty(\Omega_0)$ is a nonnegative function.

Following the proofs of Theorem \ref{exipara} and Proposition \ref{compara} (or by Theorem 1.3 in \cite{MR4913772}), we get that for any $\lambda>0$, for any $v_0\in \Wspzz\cap L^\infty(\Omega)$ a nonnegative function and for $q\geq p-1$, there exists $T_{\diamond}\in (0,+\infty]$ such that, for any $T<T_\diamond$, \eqref{8} admits a weak solution $v$ on $[0,T]\times \Omega_0$ in the sense of Definition \ref{defsol2}. In addition,  if $q\leq 1$, $v$ is a global solution and if $q\geq 1$, the uniqueness holds.
As in Proposition \ref{propReg}, we obtain that $v\in C([0,T),\Wspzz)$ and for any $t_0$, $t_1 \in [0,T)$
\begin{equation}\label{NRJ}
\int_{t_0}^{t_1} \int_{\Omega_0} |\partial_t v|^2dxdt+E_{\Omega_0}(v(t_1)) = E_{\Omega_0}(v(t_0)),
\end{equation}
where $E_{\Omega_0}$ is given by \eqref{Eomega0}. Furthermore, assuming that $E_{\Omega_0}(v_0)<0$ and $q>p-1$, Theorem 1.8 of \cite{MR4913772} implies for that the solution $v$ of \eqref{8} satisfies the following asymptotic properties with respect to $q$: 
\begin{equation*}\label{ppty}\tag{$\mathcal A _q$}
\left\{\begin{array}{l}
    \mbox{ if $q>1$, then $T_{max}<+\infty,$}\\
    \mbox{ if $q=1$, then $\|v(t)\|_{L^2(\Omega_0)}\to+ \infty$ as $t\to +\infty,$}\\
    \mbox{ if $q<1$, then $\sup_{(0,+\infty)}\|v\|_{L^\infty(\Omega_0)}=+\infty$.}
		\end{array} \right.
\end{equation*}
In the case $q=p-1$, we also get the convergence result:
\begin{proposition}
Let $q=p-1$, $\lambda>\lambda_1(\Omega_0)$ and $v_0\in \Wspzz\cap L^\infty(\Omega_0)$ such that $E_{\Omega_0}(v_0)<0$. Then, the solution $v$ of \eqref{8} satisfies $\|v(t)\|_{L^2(\Omega_0)}\to \infty$ as $t\to \infty$ if $q\leq 1$ and $T_{\max} <+\infty$ if $q>1$.
\end{proposition}
\begin{proof}
Let us begin by the case $q\leq 1$. By the energy equality \eqref{NRJ}, for any $t\in [0,T)$, $E_{\Omega_0}(v(t)) \leq  E_{\Omega_0}(v_0)$ 
and so, we take the solution $v$ as test function in the weak formulation of \eqref{8} given by \eqref{fv2} (integrating in $\Omega_0$) and the regularity of $v$ and  Proposition 1.4.29 of \cite{cazenave} imply
$$\frac12d_t\|v(t)\|^2_{L^2(\Omega_0)}=-pE_{\Omega_0}(v(t))\geq -p E_{\Omega_0}(v_0).$$
Finally, by integrating we deduce that $\|v(t)\|^2_{L^2(\Omega_0)}\to \infty$ as $t\to \infty$.
\\
For the case $q>1$, we can see \cite[Theorem 3.5]{li&xie} or \cite[Theorem B]{fujii} and adapt the method for $s\in (0,1)$. 
\end{proof}
Hence we obtain:
\begin{proof}[Proof of Proposition \ref{lemma}]
Let $\psi_{1,\Omega_0}$ be the normalized eigenfunction associated to $\lambda_1(\Omega_0)$. The function $\psi_{1,\Omega_0}$ satisfies in the weak sense
$$\pfrac \psi_{1,\Omega_0}=\lambda_1(\Omega_0)\psi_{1,\Omega_0}^{p-1}\text{ in }\Omega_0$$
and $\psi_{1,\Omega_0}\leq cd(\cdot,\Omega_0^c)$. Hence setting $v_0=\epsilon \psi_{1,\Omega_0}$,  we choose $\epsilon>0$ small enough such that $v_0\leq u_0$ and $E_{\Omega_0}(v_0)<0$. Then, the solution $v$ of \eqref{8} with $v(0,.)=v_0$ satisfies if $q\leq 1$, $\|v(t)\|^2_{L^2(\Omega_0)}\to \infty$ as $t\to\infty$ and if $q>1$, $T_{max}<+\infty$. Finally,  adapting the comparison principle in Remark \ref{compari} for the problem \eqref{8}, we conclude the result of Proposition \ref{lemma}.
\end{proof}
 We establish Theorem \ref{q>p*-1} following the same idea of the previous result.
\begin{proof}[Proof of Theorem \ref{q>p*-1}]
Let $u$ be the solution of \eqref{pbp}, then by Remark \ref{compari}, we have $u\geq v$ in $\Omega_0\times [0,T]$ where $v$ is the solution of \eqref{8} with $v(0)=v_0$ defined by the condition $u\in \mathcal H$.
Thus the  results $(i)-(iii)$ in Theorem \ref{q>p*-1} follow directly from the properties \eqref{ppty}.
\end{proof}
In order to prove Theorem \ref{exploextin} in the case $u_0\in \mathcal H_u$, we use a Sattinger type method involving the set:
$$\mathcal U_{\Omega_0}=\{v\in \Wspzz \;|\;E_{\Omega_0}(v)<m_{\Omega_0} \mbox{ and } I_{\Omega_0}(v)<0\}.$$
By straightforward computations, we have for $v\in \Wspzz\setminus\{0\}$, 
$$E_{\Omega_0}(\theta^*_0 v)=\sup_{\theta>0}E_{\Omega_0}(v\theta)$$
where
\begin{equation}\label{theta}
\theta^*_0=\theta^*_0(v)=\bigg(\frac{\|v\|^p_\Wspzz}{\lambda\|v\|^{q+1}_{L^{q+1}(\Omega_0)}}\bigg)^\frac{1}{q+1-p}\text{ satisfying }I_{\Omega_0}(\theta^*_0v)=0
\end{equation}
and hence from the definition of $m_{\Omega_0}$, we have
$$ m_{\Omega_0}=\bigg(\dfrac{1}{p}-\dfrac{1}{q+1}\bigg)\lambda^{\frac{-p}{q+1-p}}  S_0^\frac{p(q+1)}{q+1-p} \mbox{ where } S_0=\inf_{v\in \Wspzz\backslash\{0\}} \dfrac{\|v\|_\Wspzz}{\|v\|_{_{L^{q+1}(\Omega_0)}}}.$$ 
\begin{remark}
   For $sp<d$, the Sobolev embeddings imply that $S_0>0$ when $q\leq p^*-1$ and $S_0=0$ when $q> p^*-1$. In the latter case, we return to the case described in Theorem \ref{q>p*-1}. For $sp\geq d$, we have $S_0>0$ for any  $q>0$.
\end{remark}
As previously,  we compare the solution of \eqref{pbp} with the solution of \eqref{8} with ordered initial data as in the proof of Theorem \ref{q>p*-1} to prove Theorem \ref{exploextin}. For that,  we establish first the following result for the solutions of \eqref{8}:
\begin{theorem}\label{sattExplo}
Let $v_0\in \mathcal U_{\Omega_0}\cap L^\infty(\Omega_0)$ be a nonnegative function and assume that $S_0>0$. Then, the solution $v$ of \eqref{8} satisfies \eqref{ppty} and  $v(t) \in \mathcal U_{\Omega_0}$ for any $t\in [0,T_\diamond)$.
\end{theorem}
\begin{proof}
Let $T\in [0,T_\diamond)$.  We first show that if $v_0\in \mathcal U_{\Omega_0}$ then $v(t)\in \mathcal U_{\Omega_0}$ for any $t\in [0,T]$. 
From \eqref{NRJ}, we deduce $E_{\Omega_0}(v(t))\leq E_{\Omega_0}(v_0)<m_{\Omega_0}$ and $t\mapsto E_{\Omega_0}(v(t))$ is nonincreasing. 

It remains to get $I(v(t))<0$. We argue by contradiction assuming that there exists $t\leq T$ such that $I_{\Omega_0}(v(t))\geq 0$. Proposition \ref{propReg} holds for the problem \eqref{8} hence $v\in C([0,T);\Wspzz)$ and then there exists $t_0$ such that $I_{\Omega_0}(v(t_0))=0$ which implies $E_{\Omega_0}(v(t_0))\geq m_{\Omega_0}$. We get a contradiction then we conclude $v(t)\in \mathcal U_{\Omega_0}$.\\
We now study the asymptotic behavior of the solution.\\
Since $t\to E_{\Omega_0}(v(t))$ is  nonincreasing, we have from the weak formulation
\begin{equation}\label{10}
\begin{split}
\frac{1}{2}d_t\|v(t)\|^2_{L^2(\Omega_0)}&=-pE_{\Omega_0}(v(t))+\lambda(1-\frac{p}{q+1})\|v(t)\|^{q+1}_{L^{q+1}(\Omega_0)}\\
&\geq-pE_{\Omega_0}(v_0)+\lambda(1-\frac{p}{q+1})\|v(t)\|^{q+1}_{L^{q+1}(\Omega_0)}\\
&\geq -p(1-\tau)m_{\Omega_0}+\lambda(1-\frac{p}{q+1})\|v(t)\|^{q+1}_{L^{q+1}(\Omega_0)},
\end{split}
\end{equation}
where $\tau=1-\frac{E_{\Omega_0}(v_0)}{m_{\Omega_0}}>0$. From \eqref{theta}, we get:
\begin{equation}\label{11}
\begin{split}
-pm_{\Omega_0}= -p\inf_{w\in \Wspzz} E_{\Omega_0}(\theta^*_0 w)&=\lambda(\frac{p}{q+1}-1 )\inf_{w\in \Wspzz}\|\theta^*_0 w\|^{q+1}_{L^{q+1}(\Omega_0)}\\
&\geq \lambda(\frac{p}{q+1}-1 )\|v(t) \|^{q+1}_{L^{q+1}(\Omega_0)}
\end{split}
\end{equation}
since $\theta^*_0<1$ in $\mathcal U_{\Omega_0}$. 
Finally combining \eqref{10} and \eqref{11} we infer that
$$d_t\|v(t)\|^2_{L^2(\Omega_0)}\geq c\|v(t)\|^{q+1}_{L^{q+1}(\Omega_0)}$$
where $c$ is a positive constant depending on $\lambda$, $v_0$, $p$ and $q$.\\
We define the continuous function $Y(t)=\|v(t)\|_{L^2(\Omega_0)}^2$. Thus by the previous inequality, by H\"older inequality and for $q\geq1$, $Y$ satisfies:
\begin{equation}\label{3}
Y'(t)\geq cY(t)^\gamma\  \mbox{ with $\gamma=\frac{q+1}{2}$.} 
\end{equation}
For $q=1$, we have $\gamma=1$ and hence $Y(t)\geq c_0\exp(ct)\to \infty$ as $t\to \infty$, we deduce that the second property in \eqref{ppty} holds.\\
For $q>1$,  we argue by contradiction assuming $T_{max}=+\infty$. Integrating \eqref{3}, we get for any $t>0$, $Y^{1-\gamma}(t)\leq  - C(\gamma-1) t +Y^{1-\gamma}(0)$ which implies that $Y(t)<0$ for $t$ large enough and gives a contradiction.

Finally, for $q<1$, we also argue by contradiction assuming that for any $t\geq 0$, $\|v(t)\|_{L^\infty(\Omega_0)}\leq C$. Then, $Y(t) \leq C^{1-q}\|v\|^{q+1}_{L^{q+1}(\Omega_0)}$ and we obtain $Y'(t)\geq \tilde c Y(t)$. As the case $q=1$, we deduce that $Y(t)$ goes to $\infty$ as $t\to \infty$ which gives a contradiction.
\end{proof}
\begin{theorem}\label{sattExt}
Let $v_0\in \mathcal S_{\Omega_0}\cap L^\infty(\Omega_0)$ be a nonnegative function and assume that $sp>d$. Then, \eqref{8} admits a global solution  $v$ tending to $0$ in $L^m(\Omega_0)$ for any $m\in [1,+\infty)$ as $t\to +\infty$. 
\end{theorem}

\begin{proof}
let $T>0$ and let $v$ be a weak solution of \eqref{8} on $Q_T$, we note as in the proof of Theorem \ref{sattExplo} that  $v_0\in \mathcal S_{\Omega_0}$ implies $v(t)\in \mathcal S_{\Omega_0}$ for any $t\in [0,T]$.\\
Moreover, by the definition of $\mathcal S_{\Omega_0}$ and by the Sobolev embedding since $sp>d$, we deduce that there exists a constant $C>0$ independent of  $T$ such that 
\begin{equation}\label{sob}
 \sup_{t\in [0,T]}\|v(t)\|_{\Wspzz}+\sup_{t\in [0,T]}\|v(t)\|_{L^\infty(\Omega_0)}\leq C.
\end{equation}
We now prove the existence of a global solution. For $q\leq1$, we  already know the existence of a global solution of \eqref{8}. For $ q>1$, the inequality \eqref{sob}  and Remark \ref{tmax} insure that $T_{max}=+\infty$ and thus the unique solution of \eqref{8} is global.\\
Let $v$ be a global solution of \eqref{8}, then $v$ belongs to $C([0,\infty);\Wspzz)$.\\ 
 Let $T>0$, we define the sequence $(v_n)$ by  $v_n(t, \cdot)=v(t+nT,\cdot)$ for any $t\in[0,T]$.\\
By the definition of $\mathcal S_{\Omega_0}$, we have for any $t\in [0,\infty)$
$$E_{\Omega_0}(v(t))=\frac{1}{p}I_{\Omega_0}(v(t))+(\frac{1}{p}-\frac{1}{q+1})\|v(t)\|^{q+1}_{L^{q+1}(\Omega_0)}>0$$
and  the mapping $t\to E_{\Omega_0}(v(t))$ is nonincreasing  thus $E_{\Omega_0}(v)$ admits a finite limit as $t\to \infty$. Hence we deduce from \eqref{NRJ} that
\begin{equation}\label{12}
\lim_{n\to \infty} \int_{nT}^{(n+1)T}\|\partial_t v\|_{L^2(\Omega_0)}^2\,d\tau=\lim_{n\to \infty} \int_{0}^{T}\|\partial_t v_n\|_{L^2(\Omega_0)}^2\,d\tau= 0.
\end{equation}
From \eqref{sob} and \eqref{12}, $(v_n)_n$ is bounded in $L^\infty(0,T;\Wspzz)\cap L^\infty([0,T]\times\Omega_0)$ and $(\partial_tv_n)_n$ is bounded in $L^2(0,T;L^2(\Omega_0))$. Thus, using compactness arguments, there exists $w$ such that, up to a subsequence, 
\begin{equation}\label{conver}
\begin{array}{l}
v_n \overset{\ast}{\rightharpoonup} w \mbox{ in } L^\infty(0,T;\Wspzz)\\
v_n \to w \mbox{ in } C([0,{T}],L^m(\Omega)) \mbox{ for any } m\in [1,\infty)
\end{array}
\end{equation}
Indeed, the Aubin-Simon Theorem yields the strong convergence, for more details see for instance Step 4 of the proof of Theorem 3.1 of \cite{MR4913772}. 
We also have $\mbox{for any } \phi \in L^1(0,T;\Wspzz),\mbox{ for any } t\in [0,T]$:
$$\int_0^{t}\langle\pfrac v_n-\pfrac w, \phi\rangle \,dt\to 0. $$

Moreover, from \eqref{12}, we deduce that $w$ does not depend on time.\\
Let now show that $w=0$ {\it a.e.} in $\Omega$. For any $t\in [0,T]$, 
we have
$$\int_0^{t}\int_{\Omega_0}\partial_t v_n w \,dxds+\int_0^{t}\langle\pfrac v_n,w\rangle \,dt=\int_0^{t}\int_{\Omega_0}\lambda v_n^qw\,dxds,$$
and passing the limit as $n\to \infty$, \eqref{12} implies that 
$$\int_0^t I_{\Omega_0}(w) \,ds=0,$$
hence  
$I_{\Omega_0}(w)=0$. 
Assume that $w\neq 0$, then $ E_{\Omega_0}(w)\geq m_{\Omega_0}$ and in other hand, since $w$ is independent of $t$ we can  write by \eqref{conver}
\begin{equation*}
   \begin{split}
        E_{\Omega_0}(w)&= \frac1p\|w\|^p_{L^\infty(0,T;\Wspzz)}-\frac{\lambda}{q+1} \| w\|_{L^\infty(0,T;L^{q+1}(\Omega_0))}^{q+1}\\
        &\leq \liminf \frac1p \|v_n\|^p_{L^\infty(0,T;\Wspzz)}+\lim_{n\to \infty} -\frac{\lambda}{q+1}\| v_n\|_{{L^\infty(0,T;L^{q+1}(\Omega_0))}}^{q+1}.
             \end{split}
\end{equation*}
By continuity of $v$, there exists $t_{0,n}\in [0,T]$ such that $ \|v_n(t_{0,n})\|_{\Wspzz}=\sup_{t\in [0,T]} \|v_n(t)\|_{\Wspzz}$. Hence, we get 
\begin{equation*}
\begin{split}
         E_{\Omega_0}(w)&\leq \liminf \frac1p \|v_n(t_{0,n})\|^p_{\Wspzz}+\lim_{n\to \infty} -\frac{\lambda}{q+1}\| v_n(t_{0,n})\|_{L^{q+1}(\Omega_0)}^{q+1}\\
         &\leq \liminf \left( E_{\Omega_0}(v_n(t_{0,n}) \right)<m_{\Omega_0}
\end{split}         
\end{equation*}
which is absurd. Thus $w=0$.\\
Let us now consider $\widetilde w$ to be another limit point of the sequence $(v_n)$ in $C([0,T];L^m(\Omega_0)$. Then, up to a new subsequence, \eqref{conver} holds with $\widetilde w$ as limit, as do the subsequent calculations. Thus it follows that $\widetilde w = 0$ is the unique limit point of $(v_n)$.\\
Finally, the sequence $(v_n)$ converges to 0 in $C([0,T];L^m(\Omega_0))$, for any $m\in [1,\infty)$ and thus, we deduce that that $v(t)$ goes to $0$ in $L^m(\Omega_0)$ as $t\to \infty$.
\end{proof}

\begin{proof}[Proof of Theorem \ref{exploextin}]
Assume $u_0\in \mathcal H_u$, then there exists $v_0\in \mathcal U_{\Omega_0}$ such that $u_0\geq v_0$. Then, Remark \ref{compari} implies that  $u\geq v$ in $[0,T]\times\Omega_0 $ where $v$ is the solution of \eqref{8} with $v(0)=v_0$. Hence {\it (i)-(iii)} comes from Theorem \ref{sattExplo}.

Now assume  $u_0\in \mathcal H_s$, then there exists $w_0\in \mathcal S_{\Omega}$ such that $u_0\leq w_0$. Then, replacing $\Omega_0$ by $\Omega$, Theorem \ref{sattExt} yields that the global solution $w$  of \eqref{8} such that $w(0)=w_0$ in $\Omega$ converges to $0$ in $L^r(\Omega)$ for any $r\geq 1$. Finally, we use again Remark \ref{compari} to conclude the proof.
\end{proof}
\subsection*{Acknowledgment}
The authors would like to thank Professor Jacques Giacomoni for our discussions about the problem. The authors would also like to thank the anonymous referee for their careful reading of this manuscript, and for their suggestions which have helped to improve it.

\begin{appendices}
\section{ }
\begin{proposition}\cite{simon,kato}\label{inegalg}
There exist $c_1,c_2$ positive constants such that for any $\xi , \eta \in \mathbb{R}^d$:
\begin{align}\label{ineg1}\tag{\ref{inegalg}.1}
||\xi|^{p-2}\xi -|\eta|^{p-2}\eta|\leq c_1 \left\{ \begin{array}{lr}  |\xi -\eta |(|\xi|+|\eta|)^{p-2} & \text{if } p\geq 2, \\
|\xi-\eta|^{p-1} & \text{if } p \leq 2,
\end{array} \right.
\end{align}

\begin{align}\label{ineg2}\tag{\ref{inegalg}.2}
(|\xi|^{p-2}\xi -|\eta|^{p-2}\eta).(\xi-\eta) \geq c_2 \left\{ \begin{array}{lr}  |\xi -\eta |^{p} & \text{if } p\geq 2, \\
\dfrac{|\xi-\eta|^{2}}{(|\xi|+|\eta|)^{2-p}} & \text{if } p \leq 2.
\end{array} \right.
\end{align}
\end{proposition}
\begin{proposition}\label{accre}
The operator $\mathcal A$ is accretive in $L^\infty(\Omega)$ in the sense of Definition 3.1 of \cite{barbu2}.
\end{proposition}
\begin{proof}
Let $\lambda>0$ and let $u,\,v \in D(\mathcal A)$ such that $u+\lambda \mathcal A u= f$ and $v+\lambda \mathcal Av=g$. Since $u,\, v\in \Wspz\cap L^\infty(\Omega)$, by \eqref{ineg1}, $\phi=|u-v|^{m-1}(u-v)$ belong to $\Wspz \cap L^\infty(\Omega)$ for any $m>1$. We have by direct computation:
$$\langle \mathcal A u-\mathcal Av,\phi\rangle\geq 0.$$
So taking the test function $\phi$ in $u-v+\mathcal Au -\mathcal Av=f-g$, this yields 
$$\|u-v\|_{L^{m+1}(\Omega)}^{m+1}=\int_\Omega(f-g)\phi\leq \|f-g\|_{L^{m+1}(\Omega)}\|u-v\|_{L^{m+1}(\Omega)}^{m},$$
which gives, when $m\to \infty$, $\|u-v\|_{L^\infty(\Omega)}\leq \|f-g\|_{L^\infty(\Omega)}$ giving accretive in $L^\infty(\Omega)$.
\end{proof}
\end{appendices}

\end{document}